\documentclass[11pt,a4paper]{amsart}

\usepackage{mathtools}
\usepackage{amsthm}
\usepackage{amssymb}
\usepackage{hyperref}
\usepackage{caption}
\usepackage{nicematrix}
\usepackage{enumitem}
\usepackage{booktabs}
\usepackage[left=3cm, right=3cm]{geometry}

 \usepackage{tikz}
 \usepackage{tikz-cd}
 \usetikzlibrary{calc}
 \usetikzlibrary{decorations.pathmorphing}
 \usetikzlibrary{decorations.pathreplacing}

\theoremstyle{definition}
\newtheorem{definition}{Definition}[section]

\newtheorem{problem}[definition]{Problem}
\newtheorem{question}[definition]{Question}

\theoremstyle{plain}
\newtheorem{corollary}[definition]{Corollary}
\newtheorem{lemma}[definition]{Lemma}
\newtheorem{proposition}[definition]{Proposition}
\newtheorem{theorem}[definition]{Theorem}

\numberwithin{equation}{section}

\tikzset{
	vertex/.style={
		circle,
		minimum size=1.8mm,
		fill,
		inner sep=0,
		outer sep=0,
	},
	edge/.style={
		line width=.25mm,
	}
}

\tikzset{
	x=10mm,
	y=10mm,
	treenode/.style={
		circle,
		minimum size=1mm,
		inner sep=0,
		fill=gray,
	},
	edgelabel/.style={
		node font=\small,
		inner sep =.5mm,
	},
	columnlabel/.style={
		anchor=north,
	}
}

\ExplSyntaxOn

\cs_new_protected:Npn \cycle #1
{
	\seq_set_split:Nnn\l_tmpa_seq{~}{#1}
	(
	\seq_use:Nn\l_tmpa_seq{\,} % \, is the space between terms. Change if you want.
	)
}

\cs_new_protected:Npn \chain #1
{
	\seq_set_split:Nnn\l_tmpa_seq{~}{#1}
	(
	\seq_use:Nn\l_tmpa_seq{\,} % \, is the space between terms. Change if you want.
	\rangle
}

\ExplSyntaxOff

\newcommand*\defterm{\emph}

\DeclareMathOperator{\dom}{Dom}
\DeclareMathOperator{\im}{Im}

\newcommand*\girth[1]{\mathrm{girth}(#1)}
\newcommand*\knitdegree[1]{\mathrm{kd}(#1)}

\newcommand*\cliquenumber[2][]{\omega\parens[#1]{#2}}
\newcommand*\chromaticnumber[2][]{\chi\parens[#1]{#2}}

\newcommand*\commgraph[1]{\mathcal{G}(#1)}

\newcommand*\centre[1]{Z(#1)}
\newcommand*\idemp[1]{E(#1)}

\DeclarePairedDelimiter{\abs}{\lvert}{\rvert}

\DeclarePairedDelimiter{\parens}{\lparen}{\rparen}
\DeclarePairedDelimiter{\bracks}{\lbrack}{\rbrack}
\DeclarePairedDelimiter{\braces}{\{}{\}}

\DeclarePairedDelimiter{\set}{\{}{\}}
\DeclarePairedDelimiterX{\gset}[2]{\{}{\}}{\,#1:#2\,}

\newcommand*\Xn{\set{1,\ldots,n}}
\newcommand*\X[1]{\set{1,\ldots,#1}}

\newcommand*\NC{\mathit{NC}}

\newcommand*{\tr}[1]{\mathcal{T}(#1)}

\newcommand*{\psym}[1]{\mathcal{I}(#1)}

\newcommand*{\Alt}[1]{\mathcal{A}_{#1}}

\newcommand*{\Sym}[1]{\mathcal{S}_{#1}}

\newcommand*\Rees[4]{\mathcal{M}[#1; \allowbreak #2,\allowbreak #3; \allowbreak #4]}

\makeatletter
\newcommand*{\sizeddelimiter}[2]{\bBigg@{#1}#2}
\makeatother

\allowdisplaybreaks

\begin{document}

\title{Commuting graphs of inverse semigroups and completely regular semigroups}

\author{Tânia Paulista}
\address[T. Paulista]{%
Center for Mathematics and Applications (NOVA Math) \& Department of Mathematics\\
NOVA School of Science and Technology\\
NOVA University of Lisbon\\
2829--516 Caparica\\
Portugal
}
\email{%
t.paulista@campus.fct.unl.pt
}
\thanks{This work is funded by national funds through the FCT -- Fundação para a Ciência e a Tecnologia, I.P., under the scope of the projects UID/297/2025 and UID/PRR/297/2025 (Center for Mathematics and Applications - NOVA Math). The author is also funded by national funds through the FCT -- Fundação para a Ciência e a Tecnologia, I.P., under the scope of the studentship 2021.07002.BD}

\thanks{The author is thankful to her supervisors António Malheiro and Alan J. Cain for all the support, encouragement and guidance}


\begin{abstract}

The general ideal of this paper is to answer the following question: given a numerical property of commuting graphs, a class of semigroups $\mathcal{C}$ and $n\in\mathbb{N}$, is it possible to find a semigroup in $\mathcal{C}$ such that the chosen property is equal to $n$? We study this question for the classes of Clifford semigroups, inverse semigroups and completely regular semigroups. Moreover, the properties of commuting graphs we consider are the girth, clique number, chromatic number and knit degree.
\end{abstract}

\maketitle

\section{Introduction}

The commuting graph of a semigroup is a simple graph whose vertices are elements of the semigroup, and that describes commutativity of elements. An important feature of these graphs is the fact that there is a relationship between their combinatorial structure and the algebraic structure of the group/semigroup used to construct them. (As an example of this relationship we have the correspondence between maximum-order cliques of the commuting graph and maximum-order commutative subgroups/subsemigroups of the group/semigroup.) Because of this relationship, several authors have used commuting graphs as a tool to prove group/semigroup theoretical problems \cite{Commuting_graph_T_X, Importance_commuting_graphs_1, Sporadic_simple_groups, Importance_commuting_graphs_2, Importance_commuting_graphs_3, Importance_commuting_graphs_5, Importance_commuting_graphs_6, Importance_commuting_graphs_4}. Moreover, the study of commuting graphs of important groups/semigroups has also been a popular line of resarch. Among the selected groups/semigroups we highlight the symmetric group \cite{Commuting_graph_I_X, Symmetric_group, Diameter_commuting_graph_symmetric_group, Commuting_graph_symmetric_alternating_groups}, the alternating group \cite{Commuting_graph_symmetric_alternating_groups, Alternating_group}, the transformation semigroup \cite{Commuting_graph_T_X, Largest_commutative_T_X_P_X}, the symmetric inverse semigroup \cite{Commuting_graph_I_X} and the partial transformation semigroup \cite{Largest_commutative_T_X_P_X, Diameter_P_X}.

In 2008, Iranmanesh and Jafarzadeh \cite{Commuting_graph_symmetric_alternating_groups} conjectured that there is an upper bound for the diameter of the commuting graph of a (finite non-commutative) group. This conjecture motivated Araújo, Kinyon and Konieczny \cite{Commuting_graph_T_X} to test its veracity for semigroups. They ended up proving that there is no upper bound for the diameter of the commuting graph of a semigroup by establishing that, for each $n\in\mathbb{N}$ such that $n\geqslant 2$, there is a (finite non-commutative) semigroup whose commuting graph has diameter $n$. A few years later, Giudici and Parker \cite{Group_whose_commuting_graph_has_diameter_greater_than_n} demonstrated that, for each $n\in\mathbb{N}$, there is a (finite non-commutative) group whose commuting graph has diameter greater than $n$. This allowed them to conclude that Iranmanesh and Jafarzadeh's conjecture is not true. Although the conjecture is not true for groups in general, Morgan and Parker \cite{Commuting_graphs_group_trivial_upper_bound_diameter_10} demonstrated that, if the commuting graph of a (finite non-commutative) group with trivial center is connected, then its diameter is at most $10$, proving that Iranmanesh, and Jafarzadeh's conjecture is true for groups with trivial center. More recently, Cutolo \cite{Group_whose_commuting_graph_has_diameter_n} improved Giudici and Parker's result: he showed that, for each $n\in\mathbb{N}\setminus\set{1}$, there is a (finite non-commutative) group such that the diameter of its commuting graph is equal to $n$.

Other contributions were made to the problem of determining if certain properties of commuting graphs of semigroups have an upper bound. In 2011, Araújo, Kinyon and Konieczny \cite{Commuting_graph_T_X} proved that, for each $n\in\mathbb{N}\setminus\set{1,3}$, there is a semigroup of idempotents of knit degree $n$ (and so there is no upper bound for the knit degree of a semigroup). The question of the existence of a semigroup of knit degree $3$ remained open until, in 2016, Bauer and Greenfeld \cite{Graphs_that_arise_as_commuting_graphs_of_semigroups} constructed a semigroup of knit degree $3$. Moreover, Bauer and Greenfeld \cite{Graphs_that_arise_as_commuting_graphs_of_semigroups} and, independently, Giudici and Kuzma \cite{Graphs_that_arise_as_commuting_graphs_of_semigroups_2}, characterized the graphs that arise as the commuting graph of a (finite non-commutative) semigroup. In both papers it was established that the graphs that have more than one vertex, and where no vertex is adjacent to all the others, are precisely the commuting graphs of semigroups. This characterization is also crucial for the problem at hand because it can be used to conclude that, for each $n\in\mathbb{N}$, there is a semigroup whose commuting graph is the graph given by the union of the complete graphs $K_n$ and $K_1$; this union has clique and chromatic numbers equal to $n$. and that, for each $n\in\mathbb{N}$ such that $n\geqslant 3$, there is a semigroup whose commuting graph is the graph given by the union of the cycle graph $C_n$ and the complete graph $K_1$; this union has girth equal to $n$. And so the clique number, chromatic number and girth of commuting graphs of semigroups have no upper bounds.

All of these results motivate a new problem for commuting graphs: 

\begin{question}\label{question smg group}
    What are the possible values for the diameter/clique number/chromatic number/girth/knit degree of the commuting graph of a group/semigroup?
\end{question}

In \cite{Completely_simple_semigroups_paper} the present author contributed to this question by showing that the knit degree of a group is always undefined (and, thus, the set of possible values is, in some sense, empty) and that $3$ is the unique possible value for the girth of the commuting graph of a group. We note that these sets are different from the sets of possible values for the commuting graph of a semigroup, which are, respectively, $\mathbb{N}$ and $\mathbb{N}\setminus\set{1,2}$. This makes us wonder how `far away' does one need to get from groups to `recover' these sets of possible values. Following this line of reasoning leads us to reformulate Question~\ref{question smg group} in the following form:

\begin{question}\label{question classes smg}
    Given a class of semigroups $\mathcal{C}$, what are the possible values for the diameter/clique number/girth/chromatic number/knit degree of the commuting graph of a semigroup in $\mathcal{C}$?
\end{question}

The papers \cite{Completely_simple_semigroups_paper} and \cite{Completely_0-simple_paper} include contributions to Question~\ref{question classes smg}: the first paper focuses on the class of completely simple semigroups, and the second one on the class of competely $0$-simple semigroups. The goal of this paper is to answer Question~\ref{question classes smg} with respect to the classes of Clifford semigroups, inverse semigroups and completely regular semigroups. The completely regular semigroups are, more or less, `close to' groups in the sense that completely regular semigroups share various properties with groups: for examples, all the elements of a completely regular semigroup lie in some subgroup of the semigroup. The inverse semigroups are also `close to' groups: every element of an inverse semigroup has a unique inverse and the product of an element with its inverse is an idempotent. If we consider the class of Clifford semigroups --- contained simultaneously in the classes of inverse and completely regular semigroups --- we get even `closer to' groups. In Clifford semigroups, besides the properties described above for inverse and completely regular semigroups (which are also satisfied by Clifford semigroups), we have that every element commutes with its inverse and all idempotents are central elements.

We recall that Cutolo \cite{Group_whose_commuting_graph_has_diameter_n} proved that $\mathbb{N}\setminus\set{1}$ is the set of possible values for the diameter of the commuting graph of a group. It follows from the fact that groups are Clifford, inverse and completely regular semigroups that $\mathbb{N}\setminus\set{1}$ is also the set of possible values for the diameters of their commuting graphs. Given that the diameter of commuting graphs of Clifford, inverse and completely regular semigroups has been dealt with, this paper only needs to focus on the girth, clique number, chromatic number and knit degree. In general, we will see how a given property behaves for commuting graphs of groups and we will see, as we relax some group properties to obtain Clifford and then inverse and completely regular semigroups, how that influences the properties of the commuting graph of the semigroups of the new classes. More specifically, we want to see if moving from the class of groups to the class of Clifford semigroups and then to the classes of inverse and completely regular semigroups changes the set of possible values for the girth, clique number, chromatic number and knit degree. Sections~\ref{sec: girth inv smg}, \ref{sec: clique number inv smg}, \ref{sec: chromatic number inv smg} and \ref{sec: knit degree inv smg}, respectively, approach each one of those properties. As a preliminary, in Section~\ref{Preliminaries} we gather basic definitions and results concerning simple graphs, commuting graphs, completely regular semigroups, inverse semigroups and Clifford semigroups.

This paper is based on Chapters 10 and 11 of the author's Ph.D. thesis \cite{My_thesis}.

\section{Preliminaries} \label{Preliminaries}

\subsection{Graphs}\label{Subsection graphs}

A \defterm{simple graph} $G=(V,E)$ consists of a non-empty set $V$ --- whose elements are called \defterm{vertices} --- and a set $E$ --- whose elements are called \defterm{edges} --- formed by $2$-subsets of $V$. Throughout this subsection we will assume that $G=(V,E)$ is a simple graph.

Let $x$ and $y$ be vertices of $G$. If $\set{x,y}\in E$, then we say that the vertices $x$ and $y$ are \defterm{adjacent}. If $\set{x,z}\notin E$ for all $z\in V$ (that is, if $x$ is not adjacent to any other vertex), then we say that $x$ is an \defterm{isolated vertex}.

A simple graph $H=\parens{V',E'}$ is a \defterm{subgraph} of $G$ if $V'\subseteq V$ and $E'\subseteq E$. Note that, since $H$ is a simple graph, the elements of $E'$ are $2$-subsets of $V'$.

Given $V'\subseteq V$, the \defterm{subgraph induced by $V'$} is the subgraph of $G$ whose set of vertices is $V'$ and where two vertices are adjacent if and only if they are adjacent in $G$ (that is, the set of edges of the induced subgraph is $\braces{\braces{x,y}\in E: x,y\in V'}$).

A \defterm{complete graph} is a simple graph where all distinct vertices are adjacent to each other.

A \defterm{null graph} is a simple graph with no edges and where all vertices are isolated vertices.

A \defterm{path} in $G$ from a vertex $x$ to a vertex $y$ is a sequence of pairwise distinct vertices (except, possibly, $x$ and $y$) $x=x_1,x_2,\ldots,x_n=y$ such that $\braces{x_1,x_2}, \braces{x_2,x_3},\ldots, \braces{x_{n-1},x_n}$ are pairwise distinct edges of $G$. The \defterm{length} of the path is the number of edges of the path; thus, the length of our example path is $n-1$. If $n=1$, then we call the path --- which has only one vertex and whose length is $0$ --- a \defterm{trivial path}. If $x=y$ then we call the path a \defterm{cycle}. Whenever we want to mention a path, we will write that $x=x_1-x_2-\cdots-x_n=y$ is a path (instead of writing that $x=x_1,x_2,\ldots,x_n=y$ is a path).

If the graph $G$ contains cycles, then the \defterm{girth} of $G$, denoted $\girth{G}$, is the length of a shortest cycle in $G$. If $G$ contains no cycles, then $\girth{G}=\infty$.

Let $K\subseteq V$. We say that $K$ is a \defterm{clique} in $G$ if $\braces{x,y}\in E$ for all $x,y\in K$, that is, if the subgraph of $G$ induced by $K$ is complete. The \defterm{clique number} of $G$, denoted $\cliquenumber{G}$, is the size of a largest clique in $G$, that is, $\cliquenumber{G}=\max\left\{|K|: K \text{ is a clique in } G\right\}$.

The \defterm{chromatic number} of $G$, denoted $\chromaticnumber{G}$, is the minimum number of colours required to colour the vertices of $G$ in a way such that adjacent vertices have distinct colours.

\subsection{Commuting graphs}\label{Subsection commuting graphs}

In this subsection we present the definition of commuting graph of a semigroup. We also define left path and knit degree, two concepts associated with commuting graphs.

First, we recall that the \defterm{center} of a semigroup $S$ is the set
\begin{displaymath}
	\centre{S}= \gset{x\in S}{xy=yx \text{ for all } y\in S},
\end{displaymath}
of \defterm{central} elements of $S$.

The \defterm{commuting graph} of a finite non-commutative semigroup $S$, denoted $\commgraph{S}$, is the simple graph whose set of vertices is $S\setminus Z(S)$ and where two distinct vertices $x,y\in S\setminus Z(S)$ are adjacent if and only if $xy=yx$. (Note that the semigroup must be non-commutative to guarantee that the vertex set of the commuting graph is a non-empty set.)

A \defterm{left path} in $\commgraph{S}$ is a path $x_1-x_2-\cdots-x_n$ in $\commgraph{S}$ such that $x_1\neq x_n$ and $x_1x_i=x_nx_i$ for all $i\in\braces{1,\ldots,n}$. If $\commgraph{S}$ contains left paths, then the \defterm{knit degree} of $S$, denoted $\knitdegree{S}$, is the length of a shortest left path in $\commgraph{S}$.

\subsection{Completely regular semigroups}\label{Subsection comp regular}

A semigroup $S$ is a \defterm{completely regular semigroup} if it is equipped with a unary operation $^{-1}$ that satisfies, for all $x\in S$, the following conditions:
\begin{align}
	\parens{x^{-1}}^{-1}=x
	\label{preli: comp reg smg, inverse of an inverse},\\
	xx^{-1}x=x,
	\label{preli: comp reg smg, regular}\\
	x^{-1}x=xx^{-1}.
	\label{preli: comp reg smg, x commutes with inverse}
\end{align}

The next theorem provides an alternative way of defining completely regular semigroups.

\begin{theorem}\label{preli: comp reg, every element in subgroup}
	A semigroup $S$ is completely regular if and only if every element of $S$ lies in a subgroup of $S$.
\end{theorem}

In Theorem~\ref{preli: comp reg, semillatice comp simple} we have a structure theorem for completely regular semigroups, which relies on completely simple semigroups. A \defterm{completely simple semigroup} is a completely regular semigroup that is also simple (that is, whose unique ideal is the semigroup itself). Completely simple semigroups can also be described through a semigroup construction named the Rees matrix semigroup over a group (Theorem~\ref{preli: completely simple semigroup <=> Rees matrix construction}), which we define below.

Let $G$ be a group, let $I$ and $\Lambda$ be index sets, and let $P$ be a $\Lambda \times I$ matrix whose entries are elements of $G$. For each $i\in I$ and $\lambda \in \Lambda$, we denote by $p_{\lambda i}$ the $\parens{\lambda, i}$-th entry of the matrix $P$. A \defterm{Rees matrix semigroup over a group}, denoted $\Rees{G}{I}{\Lambda}{P}$, is the set $I\times G \times \Lambda$ with multiplication defined as follows
\begin{displaymath} \parens{i,x,\lambda}\parens{j,y,\mu}=\parens{i,xp_{\lambda j}y,\mu}.
\end{displaymath}

\begin{theorem}\label{preli: completely simple semigroup <=> Rees matrix construction}
	A semigroup $S$ is completely simple if and only if there exist a group $G$, index sets $I$ and $\Lambda$, and a $\Lambda \times I$ matrix $P$ whose entries are elements of $G$ such that $S\simeq \Rees{G}{I}{\Lambda}{P}$.
\end{theorem}

A \defterm{semilattice} $Y$ is a non-empty set equipped with a partial order such that, for each $\alpha,\beta\in Y$ there exists a greatest lower bound, which we denote by $\alpha\sqcap\beta$. A semigroup $S$ is a \defterm{semilattice of completely simple semigroups} if there exists a semilattice $Y$ and completely simple subsemigroups $S_{\alpha}$ of $S$, where $\alpha\in Y$, such that $\set{S_{\alpha}}_{\alpha\in Y}$ is a partition of $S$ and $S_{\alpha}S_{\beta}\subseteq S_{\alpha\sqcap\beta}$ for all $\alpha,\beta\in S$.

\begin{theorem}\label{preli: comp reg, semillatice comp simple}
	Every completely regular semigroup is a semilattice of completely simple semigroups.
\end{theorem}
 
\subsection{Inverse semigroups}\label{Subsection inverse}

A semigroup $S$ is an \defterm{inverse semigroup} if it is equipped with a unary operation $^{-1}$ that satisfies, for all $x,y\in S$, the following conditions:
\begin{align}
	\parens{x^{-1}}^{-1}=x,
	\label{preli: inv smg, inverse of an inverse}\\
	xx^{-1}x=x,
	\label{preli: inv smg, regular}\\
	\parens{xy}^{-1}=y^{-1}x^{-1},
	\label{preli: inv smg, inverse product}\\
	\parens{xx^{-1}}\parens{yy^{-1}}=\parens{yy^{-1}}\parens{xx^{-1}}.
	\label{preli: inv smg, idemp commute}
\end{align}

Theorem~\ref{preli: characterization inv smg} provides alternative ways of characterizing inverse semigroups. Before we present that theorem, we recall the notions of regular element and inverse element.

Let $S$ be a semigroup. An element $x\in S$ is \defterm{regular} if there exists $y\in S$ such that $xyx=x$. An \defterm{inverse} of an element $x\in S$ is an element $y\in S$ such that $xyx=x$ and $yxy=y$.

\begin{theorem}\label{preli: characterization inv smg}
	Let $S$ be a semigroup. The following statements are equivalent:
	\begin{enumerate}
		\item $S$ is an inverse semigroup.
		
		\item Every element of $S$ has a unique inverse.
		
		\item $S$ is regular and all its idempotents commute.
		
		%\item Every $\mathcal{L}$-class and every $\mathcal{R}$-class of $S$ contains exactly one idempotent.
	\end{enumerate}
\end{theorem}

The last result of this subsection is the Vagner--Preston Theorem, which is the analogue for inverse semigroups of Cayley's Theorem for groups. 

\begin{theorem}[Vagner--Preston Theorem]\label{preli: vagner--preston theorem inv}
	Every inverse semigroup is isomorphic to an inverse subsemigroup of $\psym{X}$, for some set $X$.
\end{theorem}

\subsection{Clifford semigroups}\label{Subsection Clifford}

A semigroup $S$ is a \defterm{Clifford semigroup} if it is equipped with a unary operation $^{-1}$ that, for all $x,y\in S$, satisfies:
\begin{align}
	\parens{x^{-1}}^{-1}=x
	\label{preli: Clifford smg, inverse of an inverse},\\
	xx^{-1}x=x
	\label{preli: Clifford smg, regular},\\
	x^{-1}x=xx^{-1}
	\label{preli: Clifford smg, x commutes with inverse},\\
	\parens{xx^{-1}}\parens{yy^{-1}}=\parens{yy^{-1}}\parens{xx^{-1}}.
	\label{preli: Clifford smg, idemp commute}
\end{align}

The theorem below gives us an alternative characterization of Clifford semigroups.

\begin{theorem}\label{preli: characterization Clifford smg}
	A semigroup $S$ is a Clifford semigroup if and only if $S$ is regular and all its idempotents are central.
\end{theorem}

It is clear that any Clifford semigroup satisfies conditions \eqref{preli: comp reg smg, inverse of an inverse}, \eqref{preli: comp reg smg, regular} and \eqref{preli: comp reg smg, x commutes with inverse}, which implies that Clifford semigroups are completely regular semigroups.  Moreover, it follows from Theorems~\ref{preli: characterization inv smg} and \ref{preli: characterization Clifford smg} that Clifford semigroups are also inverse semigroups.

\section{Possible values for the girth of commuting graphs}\label{sec: girth inv smg}

In \cite[Proposition 4.2]{Completely_simple_semigroups_paper} the present author proved that the unique possible value for the girth of the commuting graph of a group is $3$. The aim of this section is to prove that this result extends to commuting graphs of Clifford semigroups, and then to commuting graphs of inverse semigroups. Since Clifford semigroups are completely regular semigroups, then $3$ is also a possible value for the girth of the commuting graph of a completely regular semigroup. However, $3$ is not the unique possible value when we consider completely regular semigroups: we will see that for all even $n\in\mathbb{N}\setminus\set{1,2,3}$ there is a completely regular semigroup whose commuting graph has girth equal to $n$.

We begin by introducing the following lemma, which is important to establish that $3$ is the unique possible value for the girth of the commuting graph of a Clifford semigroup.

We recall that the set of idempotents of a semigroup $S$ is denoted by $\idemp{S}$.

\begin{lemma}\label{inv smg: Clifford smg, lemma girth}
	Let $S$ be a Clifford semigroup and let $x,y\in S$. Then
	\begin{enumerate}
		\item We have $xy=yx$ if and only if $x^{-1}y=yx^{-1}$.
		
		\item If $xxy=xyx$, then $xy=yx$.
	\end{enumerate}
\end{lemma}

\begin{proof}
	\textbf{Part 1.} Suppose that $xy=yx$. Due to the fact that $S$ is a Clifford semigroup, we have
	\begin{align*}
		x^{-1}y&=\parens{x^{-1}\parens{x^{-1}}^{-1}x^{-1}}y& \bracks{\text{by \eqref{preli: Clifford smg, regular}}}\\
		&=\parens{x^{-1}xx^{-1}}y& \bracks{\text{by \eqref{preli: Clifford smg, inverse of an inverse}}}\\
		&=x^{-1}\parens{xx^{-1}}y& \bracks{\text{rearranging parentheses}}\\
		&=x^{-1}y\parens{xx^{-1}}& \bracks{\text{because } xx^{-1}\in\idemp{S}\subseteq\centre{S}}\\
		&=x^{-1}\parens{yx}x^{-1}& \bracks{\text{rearranging parentheses}}\\
		&=x^{-1}\parens{xy}x^{-1}& \bracks{\text{because } xy=yx}\\
		&=\parens{x^{-1}x}yx^{-1}& \bracks{\text{rearranging parentheses}}\\
		&=y\parens{x^{-1}x}x^{-1}& \bracks{\text{because } x^{-1}x\in\idemp{S}\subseteq\centre{S}}\\
		&=y\parens{x^{-1}\parens{x^{-1}}x^{-1}}& \bracks{\text{by \eqref{preli: Clifford smg, inverse of an inverse}}}\\
		&=yx^{-1}.& \bracks{\text{by \eqref{preli: Clifford smg, regular}}}
	\end{align*}
	
	Now suppose that $x^{-1}y=yx^{-1}$. Then it follows from what we proved earlier that $\parens{x^{-1}}^{-1}y=y\parens{x^{-1}}^{-1}$ and, consequently, by \eqref{preli: Clifford smg, inverse of an inverse}, we have $xy=yx$.
	
	\medskip
	
	\textbf{Part 2:} Suppose that $xxy=xyx$. Since $S$ is a Clifford semigroup we have
	\begin{align*}
		xy&=\parens{xx^{-1}x}y& \bracks{\text{by \eqref{preli: Clifford smg, regular}}}\\
		&=\parens{xx^{-1}}xy& \bracks{\text{rearranging parentheses}}\\
		&=\parens{x^{-1}x}xy& \bracks{\text{by \eqref{preli: Clifford smg, x commutes with inverse}}}\\
		&=x^{-1}\parens{xxy}& \bracks{\text{rearranging parentheses}}\\
		&=x^{-1}\parens{xyx}& \bracks{\text{because } xxy=xyx}\\
		&=\parens{x^{-1}x}yx& \bracks{\text{rearranging parentheses}}\\
		&=yx\parens{x^{-1}x}& \bracks{\text{because } x^{-1}x\in\idemp{S}\subseteq\centre{S}}\\
		&=yx. &\bracks{\text{by \eqref{preli: Clifford smg, regular}}} & \qedhere
	\end{align*}
\end{proof}

\begin{theorem}\label{inv smg: girth Clifford smg}
	Let $S$ be a finite non-commutative Clifford semigroup. If $\commgraph{S}$ contains at least one cycle, then $\girth{\commgraph{S}}=3$.
\end{theorem}

\begin{proof}
	Let $x_1-x_2-\cdots-x_n-x_1$ be a cycle in $\commgraph{S}$. We consider two cases:
	
	\smallskip
	
	\textit{Case 1:} Suppose that there exists $i\in\Xn$ such that $x_i\neq x_i^{-1}$. Assume, without loss of generality, that $i=1$. Since $x_1\in S\setminus\centre{S}$, then there exists $y\in S$ such that $x_1y\neq yx_1$. Hence, by part 1 of Lemma~\ref{inv smg: Clifford smg, lemma girth}, $x_1^{-1}y\neq yx_1^{-1}$ and, consequently, $x_1^{-1}\in S\setminus\centre{S}$ (that is, $x_i^{-1}$ is also a vertex of $\commgraph{S}$).
	
	Let $x\in\set{x_2,x_n}\setminus\set{x_1^{-1}}$. We have $x_1x=xx_1$, which implies, by part 1 of Lemma~\ref{inv smg: Clifford smg, lemma girth}, that $x_1^{-1}x=xx_1^{-1}$. Moreover, by \eqref{preli: Clifford smg, x commutes with inverse}, we have $x_1x_1^{-1}=x_1^{-1}x_1$. Thus $x_1-x-x_1^{-1}-x_1$ is a cycle (of length 3) in $\commgraph{S}$.
	
	\smallskip
	
	\textit{Case 2:} Suppose that $x_i=x_i^{-1}$ for all $i\in\Xn$. If $x_1x_3=x_3x_1$, then $x_1-x_2-x_3-x_1$ is a cycle (of length 3) in $\commgraph{S}$.
	
	Now assume that $x_1x_3\neq x_3x_1$. We have $\parens{x_1x_3}^{-1}=x_3^{-1}x_1^{-1}=x_3x_1\neq x_1x_3$ (because $S$ is an inverse semigroup and by \eqref{preli: inv smg, inverse product}). Then $x_2\neq x_1x_3$ and $x_2\neq\parens{x_1x_3}^{-1}$ (since $x_2=x_2^{-1}$). Furthermore, we have that $x_2$ commutes with $x_1x_3$ and $\parens{x_1x_3}^{-1}=x_3x_1$ (because $x_2$ commutes with $x_1$ and $x_3$) and, by \eqref{preli: Clifford smg, x commutes with inverse}, $x_1x_3$ commutes with $\parens{x_1x_3}^{-1}$. In addition, since $x_1x_3\neq x_3x_1$, then, by part 2 of Lemma~\ref{inv smg: Clifford smg, lemma girth}, we have $x_1\parens{x_1x_3}\neq \parens{x_1x_3}x_1$ and $x_3\parens{x_1x_3}^{-1}=x_3\parens{x_3x_1}\neq \parens{x_3x_1}x_3=\parens{x_1x_3}^{-1}x_3$, which implies that $x_1x_3,\parens{x_1x_3}^{-1}\in S\setminus\centre{S}$. Therefore $x_1x_3-x_2-\parens{x_1x_3}^{-1}-x_1x_3$ is a cycle (of length 3) in $\commgraph{G}$.
	
	\smallskip
	
	In both cases presented above we established that $\commgraph{S}$ contains a cycle of length $3$. Thus $\girth{\commgraph{S}}=3$.
\end{proof}

In the lemma stated below we show two results that are required to prove that the unique possible value for the girth of the commuting graph of an inverse semigroup is $3$ (Theorem~\ref{inv smg: girth inv smg}).

\begin{lemma}\label{inv smg: lemma girth inv smg}
	Let $S$ be an inverse subsemigroup of $\psym{X}$, for some set $X$. Then
	\begin{enumerate}
		\item If $S$ contains a non-central idempotent, then there exists $\alpha\in S$ such that $\dom\alpha\neq\im\alpha$.
		
		\item If there exists $\alpha\in S$ such that $\dom\alpha\neq\im\alpha$, then $\alpha\alpha^{-1}$ and $\alpha^{-1}\alpha$ are distinct non-central idempotents of $S$ which do not commute with $\alpha$; that is, which satisfy $\alpha\parens{\alpha\alpha^{-1}}\neq \parens{\alpha\alpha^{-1}}\alpha$ and $\alpha\parens{\alpha^{-1}\alpha}\neq\parens{\alpha^{-1}\alpha}\alpha$.
	\end{enumerate}
\end{lemma}

\begin{proof}
	\textbf{Part 1.} Suppose that $S$ contains a non-central idempotent. Let $e\in S$ be such idempotent. Then there exists $\alpha\in S$ such that $e\alpha\neq\alpha e$.
	
	If $\dom\alpha\neq\im\alpha$, then the result follows. Now suppose that $\dom\alpha=\im\alpha$. We have
	\begin{displaymath}
		\alpha|_{\im e\cap\dom\alpha}=e\alpha\neq\alpha e=\alpha|_{\parens{\im\alpha\cap \dom e}\alpha^{-1}},
	\end{displaymath}
	which implies that $\im e\cap\dom\alpha\neq \parens{\im\alpha\cap \dom e}\alpha^{-1}$. Therefore
	\begin{align*}
		\im e\alpha &=\parens{\im e\cap\dom\alpha}\alpha\\
		&\neq \parens{\parens{\im\alpha\cap \dom e}\alpha^{-1}}\alpha &\bracks{\text{since }\alpha \text{ is injective}}\\
		&=\im\alpha\cap \dom e\\
		&=\dom\alpha\cap \im e &\bracks{\text{since }\dom\alpha=\im\alpha \text{ and } \dom e=\im e}\\
		&=\parens{\dom\alpha\cap \im e}e^{-1}\\
		&=\dom e\alpha,
	\end{align*}
	which concludes the proof.

	\medskip
	
	\textbf{Part 2.} Let $\alpha\in S$ be such that $\dom\alpha\neq\im\alpha$. We have that $\alpha\alpha^{-1}$ and $\alpha^{-1}\alpha$ are idempotents and $\dom\alpha\alpha^{-1}=\dom\alpha\neq\im\alpha=\dom\alpha^{-1}\alpha$. Hence $\alpha\alpha^{-1}\neq\alpha^{-1}\alpha$. Additionally, we have
	\begin{align*}
		\abs{\dom\alpha\alpha\alpha^{-1}}&=\abs{\parens{\im\alpha\cap\dom\alpha\alpha^{-1}}\alpha^{-1}}\\
		&=\abs{\im\alpha\cap\dom\alpha\alpha^{-1}}& \bracks{\text{since } \alpha \text{ is injective}}\\
		&=\abs{\im\alpha\cap\dom\alpha} &\bracks{\text{since } \dom\alpha\alpha^{-1}=\dom\alpha}\\
		&<\abs{\dom\alpha}& \bracks{\text{since } \dom\alpha\neq\im\alpha \text{ and } \abs{\dom\alpha}=\abs{\im\alpha}}\\
		&=\abs{\dom\alpha\alpha^{-1}\alpha} &\bracks{\text{by \eqref{preli: inv smg, regular}}}
	\end{align*}
	and
	\begin{align*}
		\abs{\dom\alpha\alpha^{-1}\alpha}&=\abs{\dom\alpha} &\bracks{\text{by \eqref{preli: inv smg, regular}}}\\
		&>\abs{\im\alpha\cap\dom\alpha}& \bracks{\text{since } \dom\alpha\neq\im\alpha \text{ and } \abs{\dom\alpha}=\abs{\im\alpha}}\\
		&=\abs{\im\alpha^{-1}\alpha\cap\dom\alpha}& \bracks{\text{since } \im\alpha^{-1}\alpha=\dom\alpha^{-1}\alpha=\im\alpha}\\
		&=\abs{\parens{\im\alpha^{-1}\alpha\cap\dom\alpha}\parens{\alpha^{-1}\alpha}^{-1}} \kern -7mm & \bracks{\text{since } \alpha^{-1}\alpha \text{ is injective}}\\
		&=\abs{\dom\alpha^{-1}\alpha\alpha},
	\end{align*}
	which implies that $\alpha\parens{\alpha\alpha^{-1}}\neq\parens{\alpha\alpha^{-1}}\alpha$ and $\alpha\parens{\alpha^{-1}\alpha}\neq \parens{\alpha^{-1}\alpha}\alpha$. Therefore $\alpha\alpha^{-1},\alpha^{-1}\alpha\in S\setminus\centre{S}$.
	% We have that $\abs{\dom{\alpha\alpha^{-1}\alpha}}=\abs{\dom\alpha}$ (by \eqref{preli: inv smg, regular}). Since $\im\alpha\neq\dom\alpha=\dom\alpha\alpha^{-1}$ and $\abs{\im\alpha}=\abs{\dom\alpha}=\abs{\dom\alpha\alpha^{-1}}$, then $\abs{\dom{\alpha\parens{\alpha\alpha^{-1}}}}<\abs{\dom\alpha}=\abs{\dom{\parens{\alpha\alpha^{-1}}\alpha}}$, which implies that $\alpha\parens{\alpha\alpha^{-1}}\neq \parens{\alpha\alpha^{-1}}\alpha$. Hence $\alpha\alpha^{-1}\in S\setminus \centre{S}$. Furthermore, since $\im\alpha^{-1}\alpha=\im\alpha\neq\dom\alpha$ and  $\abs{\im\alpha^{-1}\alpha}=\abs{\im\alpha}=\abs{\dom\alpha}$, then we have that $\abs{\dom{\parens{\alpha^{-1}\alpha}\alpha}}<\abs{\dom\alpha}=\abs{\dom{\alpha\parens{\alpha^{-1}\alpha}}}$, which implies that $\parens{\alpha^{-1}\alpha}\alpha\neq \alpha\parens{\alpha^{-1}\alpha}$. Hence $\alpha^{-1}\alpha\in S \setminus\centre{S}$.
\end{proof}

\begin{theorem}\label{inv smg: girth inv smg}
	Let $S$ be a finite non-commutative inverse semigroup. If $\commgraph{S}$ contains at least one cycle, then $\girth{\commgraph{S}}=3$.
\end{theorem}

\begin{proof}
	By the Vagner--Preston Theorem~\ref{preli: vagner--preston theorem inv} we can assume, without loss of generality, that $S$ is an inverse subsemigroup of $\psym{X}$, for some finite set $X$.
	
	We begin by dealing with two easy cases.
	
	Suppose that $S$ is a Clifford semigroup. Then, by Theorem~\ref{inv smg: girth Clifford smg}, $\girth{\commgraph{S}}=3$.
	
	Suppose that $S$ contains at least three non-central idempotents. Due to the fact that $S$ is an inverse semigroup, their idempotents commute and, consequently, $\commgraph{S}$ contains three vertices adjacent to each other. Thus $\girth{\commgraph{S}}=3$.
	
	Now we deal with the remaining (and more complicated) case. Suppose that $S$ is not a Clifford semigroup and $S$ contains at most two non-central idempotents. Since $S$ is regular but not a Clifford semigroup, then $S$ must contain a non-central idempotent (by Theorem~\ref{preli: characterization Clifford smg}), which guarantees (by part 1 of Lemma~\ref{inv smg: lemma girth inv smg}) the existence of $\alpha\in S$ such that $\dom\alpha\neq\im\alpha$. Then, by part 2 of Lemma~\ref{inv smg: lemma girth inv smg}, we have that $\alpha\alpha^{-1}$ and $\alpha^{-1}\alpha$ are distinct non-central idempotents of $S$ which satisfy $\alpha\parens{\alpha\alpha^{-1}}\neq \parens{\alpha\alpha^{-1}}\alpha$ and $\alpha\parens{\alpha^{-1}\alpha}\neq \parens{\alpha^{-1}\alpha}\alpha$.
	
	The fact that $\alpha\alpha^{-1}$ and $\alpha^{-1}\alpha$ are distinct non-central idempotents of $S$ implies that $\idemp{S}\setminus\centre{S}=\set{\alpha\alpha^{-1},\alpha^{-1}\alpha}$ (because $S$ contains at most two non-central idempotents).
	
	Furthermore, the fact that $\alpha\parens{\alpha\alpha^{-1}}\neq \parens{\alpha\alpha^{-1}}\alpha$ and $\alpha\parens{\alpha^{-1}\alpha}\neq \parens{\alpha^{-1}\alpha}\alpha$ implies, together with \eqref{preli: inv smg, inverse of an inverse} and \eqref{preli: inv smg, inverse product}, that
	\begin{gather*}
		\alpha^{-1}\parens{\alpha\alpha^{-1}} =\parens{\parens{\alpha\alpha^{-1}}\alpha}^{-1} \neq\parens{\alpha\parens{\alpha\alpha^{-1}}}^{-1} =\parens{\alpha\alpha^{-1}}\alpha^{-1}\\
		\shortintertext{and}
		\alpha^{-1}\parens{\alpha^{-1}\alpha} =\parens{\parens{\alpha^{-1}\alpha}\alpha}^{-1} \neq\parens{\alpha\parens{\alpha^{-1}\alpha}}^{-1} =\parens{\alpha^{-1}\alpha}\alpha^{-1}.
	\end{gather*}
	Moreover, $\alpha\alpha^{-1}$ and $\alpha^{-1}\alpha$ are idempotents and, consequently, they commute. Therefore the subgraph of $\commgraph{S}$ induced by $\set{\alpha,\alpha^{-1},\alpha\alpha^{-1},\alpha^{-1}\alpha}$ is:
	
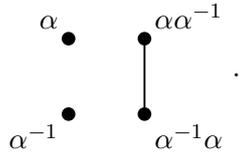
\begin{center}
		\begin{tikzpicture}[baseline=-5mm]
			
			\node[vertex] (alpha) at (0,0) {};
			\node[vertex] (alphainv) at (0,-1) {};
			\node[vertex] (idemp1) at (1,0) {};
			\node[vertex] (idemp2) at (1,-1) {};

			\node[anchor=south east] at (alpha) {$\alpha$};
			\node[anchor=north east] at (alphainv) {$\alpha^{-1}$};
			\node[anchor=south west] at (idemp1) {$\alpha\alpha^{-1}$};
			\node[anchor=north west] at (idemp2) {$\alpha^{-1}\alpha$};

			\draw[edge] (idemp1)--(idemp2);
			
		\end{tikzpicture}.
	\end{center}
	
	Since the subgraph of $\commgraph{S}$ induced by $\set{\alpha,\alpha^{-1},\alpha\alpha^{-1},\alpha^{-1}\alpha}$ does not contain cycles, then this means that $\commgraph{S}$ contains vertices which are not vertices of that subgraph. Hence $A=\parens{S\setminus\centre{S}}\setminus\set{\alpha,\alpha^{-1},\alpha\alpha^{-1},\alpha^{-1}\alpha}\neq\emptyset$. We consider the following scenarios:
	\begin{enumerate}
		\item There exists $\beta\in A$ such that $\dom\beta\neq\im\beta$.
		
		\item There exists $\beta\in A$ such that $\dom\beta=\im\beta$ and $\beta\beta^{-1}\in S\setminus\centre{S}$.
		
		\item There exists $\beta\in A$ such that $\dom\beta=\im\beta$ and $\dom\alpha\cup\im\alpha\nsubseteq\dom\beta$ and $\beta\beta^{-1}\in \centre{S}$.
		
		\item There exists $\beta\in A$ such that $\dom\beta=\im\beta$ and $\dom\alpha\cup\im\alpha\subseteq\dom\beta$ and $\parens{\dom\alpha}\beta=\dom\alpha$ and $\parens{\im\alpha}\beta=\im\alpha$ and $\beta\beta^{-1}\in \centre{S}$.
		
		\item None of the conditions above holds.
	\end{enumerate}
	It is clear that at least one of these cases holds. We are going to see that each case implies the existence of a cycle of length $3$ in $\commgraph{S}$.
	
	%there exists $\beta \in T\setminus\centre{T}$ such that $\beta\notin\set{\alpha,\alpha^{-1},\alpha\alpha^{-1},\alpha^{-1}\alpha}=\set{\alpha,\alpha^{-1}}\cap\idemp{T}$. We consider the following two sub-cases.
	
	%As a consequence of $\idemp{T}=\set{\alpha,\alpha^{-1},\alpha\alpha^{-1},\alpha^{-1}\alpha}$, then $\beta\notin\idemp{T}$.
	
	Before we analyse each one of the cases described above, we notice that due to the fact that $\dom\alpha\alpha^{-1}=\dom\alpha\neq\im\alpha=\dom\alpha^{-1}\alpha$ and $\abs{\dom\alpha\alpha^{-1}}=\abs{\dom\alpha^{-1}\alpha}$ we have that
	\begin{displaymath}
		\abs{\dom\parens{\alpha\alpha^{-1}}\parens{\alpha^{-1}\alpha}} =\abs{\dom\alpha\alpha^{-1}\cap\dom\alpha^{-1}\alpha} <\abs{\dom\alpha\alpha^{-1}} =\abs{\dom\alpha^{-1}\alpha},
	\end{displaymath}
	which implies that $\parens{\alpha\alpha^{-1}}\parens{\alpha^{-1}\alpha}\notin\set{\alpha\alpha^{-1},\alpha^{-1}\alpha}=\idemp{S}\setminus\centre{S}$ and, consequently, $\parens{\alpha\alpha^{-1}}\parens{\alpha^{-1}\alpha}\in\idemp{S}\cap\centre{S}$.
	
	\smallskip
	
	\textit{Case 1}: Suppose that condition 1 holds. Let $\beta\in A$ be such that $\dom\beta\neq\im\beta$. It follows from part 2 of Lemma~\ref{inv smg: lemma girth inv smg} that $\beta\beta^{-1}\neq\beta^{-1}\beta$ and $\beta\beta^{-1},\beta^{-1}\beta\in \idemp{S}\setminus\centre{S}=\set{\alpha\alpha^{-1},\alpha^{-1}\alpha}$. Assume, without loss of generality, that $\beta\beta^{-1}=\alpha\alpha^{-1}$ and $\beta^{-1}\beta=\alpha^{-1}\alpha$. Then
	\begin{gather*}
		\dom\beta=\dom\beta\beta^{-1}=\dom\alpha\alpha^{-1}=\dom\alpha\\
		\shortintertext{and}
		\im\beta=\dom\beta^{-1}\beta=\dom\alpha^{-1}\alpha=\im\alpha.
	\end{gather*}
	
	We are going to see that $\beta\alpha^{-1}$ is a non-central element of $S$, distinct from and commuting with $\alpha\alpha^{-1}$ and $\alpha^{-1}\alpha$.
	
	We have
	\begin{align*}
		\abs{\dom\beta\alpha^{-1}\alpha} &=\abs{\parens{\im\beta\cap\dom\alpha^{-1}\alpha}\beta^{-1}} \kern -5mm & \\
		&=\abs{\im\beta\cap\dom\alpha^{-1}\alpha} &\bracks{\text{because } \beta \text{ is injective}}\\
		&=\abs{\im\alpha} &\bracks{\text{because } \im\beta=\im\alpha=\dom\alpha^{-1}\alpha}\\
		&>\abs{\im\alpha\cap\dom\alpha}& \bracks{\text{because } \im\alpha\neq\dom\alpha \text{ and } \abs{\im\alpha}=\abs{\dom\alpha}}\\
		&\geqslant\abs{\im\alpha\cap\dom\beta\alpha^{-1}}& \bracks{\text{because } \dom\beta\alpha^{-1}\subseteq\dom\beta=\dom\alpha}\\
		&=\abs{\parens{\im\alpha\cap\dom\beta\alpha^{-1}}\alpha^{-1}} \kern -5mm &\bracks{\text{because } \alpha \text{ is injective}}\\
		&=\abs{\dom\alpha\beta\alpha^{-1}},
	\end{align*}
	which implies that $\dom\parens{\beta\alpha^{-1}}\alpha\neq\dom\alpha\parens{\beta\alpha^{-1}}$ and, consequently, $\parens{\beta\alpha^{-1}}\alpha\neq\alpha\parens{\beta\alpha^{-1}}$. Thus $\beta\alpha^{-1}\in S\setminus\centre{S}$.
	
	Additionally, we have
	\begin{align*}
		\parens{\beta\alpha^{-1}}\parens{\alpha^{-1}\alpha}
		&=\parens{\beta\parens{\alpha^{-1}\parens{\alpha^{-1}}^{-1}\alpha^{-1}}}\parens{\alpha^{-1}\alpha} \kern -5mm &\bracks{\text{by \eqref{preli: inv smg, regular}}} \\
		&=\parens{\beta\parens{\alpha^{-1}\alpha\alpha^{-1}}}\parens{\alpha^{-1}\alpha} &\bracks{\text{by \eqref{preli: inv smg, inverse of an inverse}}} \\
		&=\parens{\beta\alpha^{-1}}\parens{\alpha\alpha^{-1}}\parens{\alpha^{-1}\alpha} &\bracks{\text{rearranging parentheses}}\\
		&=\parens{\alpha\alpha^{-1}}\parens{\alpha^{-1}\alpha}\parens{\beta\alpha^{-1}} &\bracks{\text{because } \parens{\alpha\alpha^{-1}}\parens{\alpha^{-1}\alpha}\in\centre{S}}\\
		&=\parens{\alpha^{-1}\alpha}\parens{\alpha\alpha^{-1}} \parens{\beta\alpha^{-1}} &\bracks{\text{because } \alpha\alpha^{-1},\alpha^{-1}\alpha\in\idemp{S} \text{ commute}}\\
		&=\parens{\alpha^{-1}\alpha}\parens{\beta\beta^{-1}} \parens{\beta\alpha^{-1}} &\bracks{\text{because } \beta\beta^{-1}=\alpha\alpha^{-1}}\\
		&=\parens{\alpha^{-1}\alpha}\parens{\beta\beta^{-1} \beta}\alpha^{-1} &\bracks{\text{rearranging parentheses}}\\
		&=\parens{\alpha^{-1}\alpha}\parens{\beta\alpha^{-1}} &\bracks{\text{by \eqref{preli: inv smg, regular}}}
	\end{align*}
	and
	\begin{align*}
		\parens{\beta\alpha^{-1}}\parens{\alpha\alpha^{-1}} &=\beta\parens{\alpha^{-1}\alpha\alpha^{-1}} &\bracks{\text{rearranging parentheses}} \\
		&=\beta\parens{\alpha^{-1}\parens{\alpha^{-1}}^{-1}\alpha^{-1}} &\bracks{\text{by \eqref{preli: inv smg, inverse of an inverse}}}\\
		&=\beta\alpha^{-1} &\bracks{\text{by \eqref{preli: inv smg, regular}}}\\
		&=\parens{\beta\beta^{-1}\beta}\alpha^{-1} &\bracks{\text{by \eqref{preli: inv smg, regular}}}\\
		&=\parens{\beta\beta^{-1}}\parens{\beta\alpha^{-1}} &\bracks{\text{rearranging parentheses}} \\
		&=\parens{\alpha\alpha^{-1}}\parens{\beta\alpha^{-1}}. &\bracks{\text{because } \beta\beta^{-1}=\alpha\alpha^{-1}}
	\end{align*}
	
	Since we have
	\begin{align*}
		\alpha &\neq\beta &\bracks{\text{because } \beta\in A}\\
		&=\beta\beta^{-1}\beta &\bracks{\text{by \eqref{preli: inv smg, regular}}}\\
		&=\beta\alpha^{-1}\alpha, &\bracks{\text{because } \beta^{-1}\beta=\alpha^{-1}\alpha}
	\end{align*}
	then $\alpha\alpha^{-1}\neq\beta\alpha^{-1}$ (because, otherwise, we could use \eqref{preli: inv smg, inverse of an inverse} to obtain $\alpha=\alpha\alpha^{-1}\alpha=\beta\alpha^{-1}\alpha$, which is not possible). Furthermore, we have
	\begin{align*}
		\dom\beta\alpha^{-1}&=\parens{\im\beta\cap\dom\alpha^{-1}}\beta^{-1} \kern -8mm\\
		&=\parens{\im\beta}\beta^{-1} &\bracks{\text{because } \im\beta=\im\alpha=\dom\alpha^{-1}}\\
		&=\dom\beta\\
		&=\dom\alpha\alpha^{-1}\\
		&\neq\dom\alpha^{-1}\alpha, &\bracks{\text{because } \alpha\alpha^{-1},\alpha^{-1}\alpha\in\idemp{S} \text{ are distinct}}
	\end{align*}
	which implies that $\beta\alpha^{-1}\neq\alpha^{-1}\alpha$.
	
	We have proved that the distinct vertices $\alpha\alpha^{-1}$, $\alpha^{-1}\alpha$ and $\beta\alpha^{-1}$ of $\commgraph{S}$ are adjacent to each other. Therefore, $\alpha\alpha^{-1}-\beta\alpha^{-1}-\alpha^{-1}\alpha-\alpha\alpha^{-1}$ is a cycle (of length $3$) in $\commgraph{S}$ and, consequently, $\girth{\commgraph{S}}=3$.
	
	\smallskip
	
	\textit{Case 2:} Suppose that condition 2 holds. Let $\beta\in A$ be such that $\dom\beta=\im\beta$ and $\beta\beta^{-1}\in S\setminus\centre{S}$. We note that we have $\beta\beta^{-1}=\beta^{-1}\beta$ because $\dom\beta=\im\beta$. We have $\beta\beta^{-1}=\beta^{-1}\beta\in\idemp{S}\setminus\centre{S}=\set{\alpha\alpha^{-1},\alpha^{-1}\alpha}$. Assume, without loss of generality, that $\beta\beta^{-1}=\beta^{-1}\beta=\alpha\alpha^{-1}$.
	
	We have
	\begin{align*}
		\parens{\alpha\alpha^{-1}}\beta&=\parens{\beta\beta^{-1}}\beta &\bracks{\text{because } \beta\beta^{-1}=\alpha\alpha^{-1}}\\
		&=\beta\parens{\beta^{-1}\beta} &\bracks{\text{rearranging parentheses}}\\
		&=\beta\parens{\alpha\alpha^{-1}} &\bracks{\text{because } \beta^{-1}\beta=\alpha\alpha^{-1}}
	\end{align*}
	and
	\begin{align*}
		\beta\parens{\alpha^{-1}\alpha}&=\parens{\beta\beta^{-1}\beta}\parens{\alpha^{-1}\alpha} &\bracks{\text{by \eqref{preli: inv smg, regular}}}\\
		&=\beta\parens{\beta^{-1}\beta}\parens{\alpha^{-1}\alpha} &\bracks{\text{rearranging parentheses}}\\
		&=\beta\parens{\alpha\alpha^{-1}}\parens{\alpha^{-1}\alpha} &\bracks{\text{because } \beta^{-1}\beta=\alpha\alpha^{-1}}\\
		&=\parens{\alpha\alpha^{-1}}\parens{\alpha^{-1}\alpha}\beta &\bracks{\text{because } \parens{\alpha\alpha^{-1}}\parens{\alpha^{-1}\alpha}\in\centre{S}}\\
		&=\parens{\alpha^{-1}\alpha}\parens{\alpha\alpha^{-1}}\beta &\bracks{\text{because } \alpha\alpha^{-1},\alpha^{-1}\alpha\in\idemp{S} \text{ commute}}\\
		&=\parens{\alpha^{-1}\alpha}\parens{\beta\beta^{-1}}\beta &\bracks{\text{because } \beta\beta^{-1}=\alpha\alpha^{-1}}\\
		&=\parens{\alpha^{-1}\alpha}\parens{\beta\beta^{-1}\beta} &\bracks{\text{rearranging parentheses}}\\
		&=\parens{\alpha^{-1}\alpha}\beta. &\bracks{\text{by \eqref{preli: inv smg, regular}}}
	\end{align*}
	
	Since $\alpha\alpha^{-1}$, $\alpha^{-1}\alpha$ and $\beta$ are distinct non-central elements of $S$ that commute, then we have that $\alpha\alpha^{-1}-\beta-\alpha^{-1}\alpha-\alpha\alpha^{-1}$ is a cycle (of length $3$) in $\commgraph{S}$. Therefore $\girth{\commgraph{S}}=3$.
	
	\smallskip
	
	\textit{Case 3:} Suppose that condition 3 holds. Let $\beta\in A$ be such that $\dom\beta=\im\beta$ and $\dom\alpha\cup\im\alpha\nsubseteq\dom\beta$ and $\beta\beta^{-1}\in \centre{S}$. It follows from the fact that $\dom\beta=\im\beta$ that $\beta\beta^{-1}=\beta^{-1}\beta$.
	
	We have that
	\begin{align*}
		\abs{\dom\parens{\beta\beta^{-1}}\parens{\alpha\alpha^{-1}}} &=\abs{\parens{\im\beta\beta^{-1}\cap\dom\alpha\alpha^{-1}}\parens{\beta\beta^{-1}}^{-1}}\\
		&=\abs{\parens{\im\beta\beta^{-1}\cap\dom\alpha}\parens{\beta\beta^{-1}}^{-1}}\\
		&=\abs{\dom\parens{\beta\beta^{-1}}\alpha}\\
		&=\abs{\dom\alpha\parens{\beta\beta^{-1}}} &\bracks{\text{since } \beta\beta^{-1}\in\centre{S}}\\
		&=\abs{\parens{\im\alpha\cap\dom\beta\beta^{-1}}\alpha^{-1}}\\
		&=\abs{\im\alpha\cap\dom\beta\beta^{-1}} &\bracks{\text{since } \alpha \text{ is injective}}\\
		&=\abs{\dom\alpha^{-1}\alpha\cap\dom\beta\beta^{-1}}\\
		&=\abs{\dom\parens{\beta\beta^{-1}}\parens{\alpha^{-1}\alpha}},
	\end{align*}
	which implies that
	\begin{align*}
		\dom\alpha\nsubseteq\dom\beta &\iff \abs{\dom\beta\cap\dom\alpha}<\abs{\dom\alpha}\\
		&\iff \abs{\dom\beta\beta^{-1}\cap\dom\alpha\alpha^{-1}}<\abs{\dom\alpha}\\
		&\iff \abs{\dom\parens{\beta\beta^{-1}}\parens{\alpha\alpha^{-1}}}<\abs{\dom\alpha}\\
		&\iff \abs{\dom\parens{\beta\beta^{-1}}\parens{\alpha^{-1}\alpha}}<\abs{\dom\alpha} \\
		&\iff \abs{\dom\beta\beta^{-1}\cap\dom\alpha^{-1}\alpha}<\abs{\dom\alpha}\\
		&\iff \abs{\dom\beta\cap\im\alpha}<\abs{\dom\alpha}\\
		&\iff \abs{\dom\beta\cap\im\alpha}<\abs{\im\alpha} &\bracks{\text{because } \abs{\dom\alpha}=\abs{\im\alpha}}\\
		&\iff \im\alpha\nsubseteq\dom\beta.
	\end{align*}
	Furthermore, the fact that $\dom\alpha\cup\im\alpha\nsubseteq\dom\beta$ implies that $\dom\alpha\nsubseteq\dom\beta$ or $\im\alpha\nsubseteq\dom\beta$. Hence we must have 
	\begin{displaymath}
		\abs{\dom\parens{\beta\beta^{-1}}\parens{\alpha\alpha^{-1}}} =\abs{\dom\parens{\beta\beta^{-1}}\parens{\alpha^{-1}\alpha}} <\abs{\dom\alpha} =\abs{\dom\alpha\alpha^{-1}} =\abs{\dom\alpha^{-1}\alpha},
	\end{displaymath}
	which implies that $\parens{\beta\beta^{-1}}\parens{\alpha\alpha^{-1}},\parens{\beta\beta^{-1}}\parens{\alpha^{-1}\alpha}\notin\set{\alpha\alpha^{-1},\alpha^{-1}\alpha}=\idemp{S}\setminus\centre{S}$ and, consequently, $\parens{\beta\beta^{-1}}\parens{\alpha\alpha^{-1}},\parens{\beta\beta^{-1}}\parens{\alpha^{-1}\alpha}\in\idemp{S}\cap\centre{S}$. Therefore
	\begin{align*}
		\beta\parens{\alpha\alpha^{-1}} &= \parens{\beta\beta^{-1}\beta}\parens{\alpha\alpha^{-1}} &\bracks{\text{by \eqref{preli: inv smg, regular}}}\\
		&=\beta\parens{\beta\beta^{-1}}\parens{\alpha\alpha^{-1}} &\bracks{\text{since } \beta\beta^{-1}=\beta^{-1}\beta}\\
		&=\parens{\beta\beta^{-1}}\parens{\alpha\alpha^{-1}}\beta &\bracks{\text{since } \parens{\beta\beta^{-1}}\parens{\alpha\alpha^{-1}}\in\centre{S}}\\
		&=\parens{\alpha\alpha^{-1}}\parens{\beta\beta^{-1}}\beta &\bracks{\text{by \eqref{preli: inv smg, idemp commute}}}\\
		&=\parens{\alpha\alpha^{-1}}\beta &\bracks{\text{by \eqref{preli: inv smg, regular}}}
	\end{align*}
	and we can verify in a similar way that $\beta\parens{\alpha^{-1}\alpha}=\parens{\alpha^{-1}\alpha}\beta$.
	
	It follows from the fact that $\beta\in A$ that $\beta$ is a non-central element of $S$ distinct from $\alpha\alpha^{-1}$ and $\alpha^{-1}\alpha$. Therefore $\alpha\alpha^{-1}-\beta-\alpha^{-1}\alpha-\alpha\alpha^{-1}$ is a cycle (of length $3$) in $\commgraph{S}$ and, consequently, $\girth{\commgraph{S}}=3$.
	
	\smallskip
	
	\textit{Case 4:} Suppose that condition 4 holds. Let $\beta\in A$ be such that $\dom\beta=\im\beta$ and $\dom\alpha\cup\im\alpha\subseteq\dom\beta$ and $\parens{\dom\alpha}\beta=\dom\alpha$ and $\parens{\im\alpha}\beta=\im\alpha$ and $\beta\beta^{-1}\in \centre{S}$.
	
	We have that
	\begin{align*}
		\im\alpha\beta &=\parens{\im\alpha\cap\dom\beta}\beta\\
		&=\parens{\im\alpha}\beta &\bracks{\text{since } \im\alpha\subseteq\dom\beta}\\
		&=\im\alpha.
	\end{align*}
	This implies that the idempotents $\parens{\alpha\beta}^{-1}\parens{\alpha\beta}$ and $\alpha^{-1}\alpha$ have the same domain. Hence $\parens{\alpha\beta}^{-1}\parens{\alpha\beta}=\alpha^{-1}\alpha$. Similarly, we can use the fact that $\dom\alpha\subseteq\dom\beta$ and $\parens{\dom\alpha}\beta=\dom\alpha$ to conclude that $\im\alpha^{-1}\beta=\dom\alpha$ and, consequently, that $\parens{\alpha^{-1}\beta}^{-1}\parens{\alpha^{-1}\beta}=\alpha\alpha^{-1}$.
	% and
	% \begin{align*}
		%     \im\alpha^{-1}\beta &=\parens{\im\alpha^{-1}\cap\dom\beta}\beta\\
		%     &=\parens{\dom\alpha\cap\dom\beta}\beta\\
		%     &=\parens{\dom\alpha}\beta &\bracks{\text{since } \dom\alpha\subseteq\dom\beta}\\
		%     &=\dom\alpha, &\bracks{\text{since } \parens{\dom\alpha}\beta=\dom\alpha},
		% \end{align*}
	% which implies that the idempotents $\parens{\alpha\beta}^{-1}\parens{\alpha\beta}$ and $\alpha^{-1}\alpha$ have the same domain and the idempotents $\parens{\alpha^{-1}\beta}^{-1}\parens{\alpha^{-1}\beta}$ and $\alpha\alpha^{-1}$ also have the same domain. Hence $\parens{\alpha\beta}^{-1}\parens{\alpha\beta}=\alpha^{-1}\alpha$ and $\parens{\alpha^{-1}\beta}^{-1}\parens{\alpha^{-1}\beta}=\alpha\alpha^{-1}$.
	
	Therefore
	% \begin{align*}
		%     \parens{\alpha\alpha^{-1}}\beta &=\parens{\alpha\alpha^{-1}}\parens{\beta\beta^{-1}\beta} &\bracks{\text{by \eqref{preli: inv smg, regular}}}\\
		%     &=\parens{\beta\beta^{-1}}\parens{\alpha\alpha^{-1}}\beta &\bracks{\text{since } \beta\beta^{-1}\in\centre{T}}\\
		%     &= \beta\parens{\beta^{-1}\alpha}\parens{\alpha^{-1}\beta} &\bracks{\text{rearranging parentheses}}\\
		%     &=\beta\parens{\beta^{-1}\parens{\alpha^{-1}}^{-1}}\parens{\alpha^{-1}\beta} &\bracks{\text{by \eqref{preli: inv smg, inverse of an inverse}}}\\
		%     &=\beta\parens{\alpha^{-1}\beta}^{-1}\parens{\alpha^{-1}\beta} &\bracks{\text{by \eqref{preli: inv smg, inverse product}}}\\
		%     &=\beta\parens{\alpha\alpha^{-1}} &\bracks{\text{since } \parens{\alpha^{-1}\beta}^{-1}\parens{\alpha^{-1}\beta}=\alpha\alpha^{-1}}
		% \end{align*}
	% and
	\begin{align*}
		\parens{\alpha^{-1}\alpha}\beta &= \parens{\alpha^{-1}\alpha}\parens{\beta\beta^{-1}\beta} &\bracks{\text{by \eqref{preli: inv smg, regular}}}\\
		&=\parens{\beta\beta^{-1}}\parens{\alpha^{-1}\alpha}\beta &\bracks{\text{since } \beta\beta^{-1}\in\centre{S}}\\
		&= \beta\parens{\beta^{-1}\alpha^{-1}}\parens{\alpha\beta} &\bracks{\text{rearranging parentheses}}\\
		&=\beta\parens{\alpha\beta}^{-1}\parens{\alpha\beta} &\bracks{\text{by \eqref{preli: inv smg, inverse product}}}\\
		&=\beta\parens{\alpha^{-1}\alpha} &\bracks{\text{since } \parens{\alpha\beta}^{-1}\parens{\alpha\beta}=\alpha^{-1}\alpha}
	\end{align*}
	and we can check in a near-identical way that $\parens{\alpha\alpha^{-1}}\beta=\beta\parens{\alpha\alpha^{-1}}$.
	
	Due to the fact that $\beta\in\parens{S\setminus\centre{S}}\setminus\set{\alpha\alpha^{-1},\alpha^{-1}\alpha}$ we can conclude that $\alpha\alpha^{-1}-\beta-\alpha^{-1}\alpha-\alpha\alpha^{-1}$ is a cycle (of length $3$) in $\commgraph{S}$. Therefore $\girth{\commgraph{S}}=3$.
	
	\smallskip
	
	\textit{Case 5:} Suppose that conditions 1--4 do not hold. Let
	\begin{displaymath}
		M=A\cup\centre{S}=S\setminus\set{\alpha,\alpha^{-1},\alpha\alpha^{-1},\alpha^{-1}\alpha}.
	\end{displaymath}
	Our aim is to demonstrate that $M$ is a Clifford semigroup. In order to do this, we first prove the following statements:
	\begin{enumerate}
		\item[a)] $\dom\beta=\im\beta$ for all $\beta\in A$.
		\item[b)] $\beta\beta^{-1}\in\centre{S}$ for all $\beta\in A$. 
		\item[c)] $\dom\alpha\cup\im\alpha\subseteq\dom\beta$ for all $\beta\in A$.
		\item[d)] $\parens{\dom\alpha}\beta=\im\alpha$ for all $\beta\in A$.
		\item[e)] $\parens{\im\alpha}\beta=\dom\alpha$ for all $\beta\in A$.
		\item[f)] $\beta\gamma\neq\gamma\beta$ for all $\beta\in A$ and $\gamma\in\set{\alpha,\alpha^{-1},\alpha\alpha^{-1},\alpha^{-1}\alpha}$.
	\end{enumerate}
	
	Let $\beta\in A$. Since condition 1 does not hold, then we have $\dom\beta=\im\beta$; since condition 2 does not hold, then we have $\beta\beta^{-1}\in\centre{S}$; and, since condition 3 does not hold, then we have $\dom\alpha\cup\im\alpha\subseteq\dom\beta$. This proves that conditions a), b) and c) hold.
	
	Now we are going to see that conditions d) and e) hold. Let $\beta\in A$. Let $\gamma_1=\alpha\alpha^{-1}\beta$ and $\gamma_2=\alpha^{-1}\alpha\beta$. It follows from the fact that condition 4 does not hold --- and the fact that conditions a), b) and c) do --- that $\parens{\dom\alpha}\beta\neq\dom\alpha$ or $\parens{\im\alpha}\beta\neq\im\alpha$.
	
	By condition c) we have that $\dom\alpha\subseteq\dom\beta$ and, consequently,
	\begin{multline*}
		\dom\gamma_1=\dom\alpha\alpha^{-1}\beta=\parens{\im\alpha\alpha^{-1}\cap\dom\beta}\parens{\alpha\alpha^{-1}}^{-1}=\parens{\dom\alpha\cap\dom\beta}\parens{\alpha\alpha^{-1}}^{-1}\\
		=\parens{\dom\alpha}\parens{\alpha\alpha^{-1}}^{-1}=\parens{\im\alpha\alpha^{-1}}\parens{\alpha\alpha^{-1}}^{-1}=\dom\alpha\alpha^{-1}=\dom\alpha
	\end{multline*}
%	\begin{align*}
%		\dom\gamma_1&=\dom\alpha\alpha^{-1}\beta\\
%		&=\parens{\im\alpha\alpha^{-1}\cap\dom\beta}\parens{\alpha\alpha^{-1}}^{-1}\\
%		&=\parens{\dom\alpha\cap\dom\beta}\parens{\alpha\alpha^{-1}}^{-1}\\
%		&=\parens{\dom\alpha}\parens{\alpha\alpha^{-1}}^{-1}\\
%		&=\parens{\im\alpha\alpha^{-1}}\parens{\alpha\alpha^{-1}}^{-1}\\
%		&=\dom\alpha\alpha^{-1}\\
%		&=\dom\alpha
%	\end{align*}
	and
	\begin{displaymath}
		\im\gamma_1=\im\alpha\alpha^{-1}\beta=\parens{\im\alpha\alpha^{-1}\cap\dom\beta}\beta=\parens{\dom\alpha\cap\dom\beta}\beta=\parens{\dom\alpha}\beta.
	\end{displaymath}
	Since, by condition c), we also have $\im\alpha\subseteq\dom\beta$, then we can prove in an analogous way that $\dom\gamma_2=\im\alpha$ and $\im\gamma_2=\parens{\im\alpha}\beta$.
	
	It follows from the fact that $\dom\gamma_1=\dom\alpha$ and $\dom\gamma_2=\im\alpha$ that the domains of the idempotents $\gamma_1\gamma_1^{-1}$ and $\gamma_2\gamma_2^{-1}$ are equal to the domains of the idempotents $\alpha\alpha^{-1}$ and $\alpha^{-1}\alpha$, respectively. Hence $\gamma_1\gamma_1^{-1}=\alpha\alpha^{-1}$ and $\gamma_2\gamma_2^{-1}=\alpha^{-1}\alpha$.
	
	Furthermore, we have $\im\gamma_1=\parens{\dom\alpha}\beta\neq\dom\alpha=\dom\gamma_1$ or $\im\gamma_2=\parens{\im\alpha}\beta\neq\im\alpha=\dom\gamma_2$. Then, by part 2 of Lemma~\ref{inv smg: lemma girth inv smg}, we have that $\gamma_1^{-1}\gamma_1\in \idemp{S}\setminus\centre{S}=\set{\alpha\alpha^{-1},\alpha^{-1}\alpha}$ or $\gamma_2^{-1}\gamma_2\in \idemp{S}\setminus\centre{S}=\set{\alpha\alpha^{-1},\alpha^{-1}\alpha}$. Part 2 of Lemma~\ref{inv smg: lemma girth inv smg} also implies that $\gamma_1\gamma_1^{-1}\neq\gamma_1^{-1}\gamma_1$ or $\gamma_2\gamma_2^{-1}\neq\gamma_2^{-1}\gamma_2$ and, consequently, that $\gamma_1^{-1}\gamma_1=\alpha^{-1}\alpha$ or $\gamma_2^{-1}\gamma_2=\alpha\alpha^{-1}$. Thus we have either
	\begin{gather*}
		\parens{\dom\alpha}\beta=\im\gamma_1=\im\gamma_1^{-1}\gamma_1=\im\alpha^{-1}\alpha=\im\alpha\\
		\shortintertext{or}
		\parens{\im\alpha}\beta=\im\gamma_2=\im\gamma_2^{-1}\gamma_2=\im\alpha\alpha^{-1}=\dom\alpha.
	\end{gather*}
	
	To summarize, we just proved the following: if $\parens{\dom\alpha}\beta\neq\dom\alpha$, then $\parens{\dom\alpha}\beta=\im\alpha$; and, if $\parens{\im\alpha}\beta\neq\im\alpha$, then $\parens{\im\alpha}\beta=\dom\alpha$. Therefore we have
	\begin{align*}
		\parens{\dom\alpha}\beta\neq\dom\alpha &\implies \parens{\dom\alpha}\beta=\im\alpha\\
		&\implies \parens{\im\alpha}\beta\neq\im\alpha & \bracks{\text{since } \dom\alpha\neq\im\alpha \text{ and } \beta \text{ is injective}}\\
		&\implies \parens{\im\alpha}\beta=\dom\alpha\\
		&\implies \parens{\dom\alpha}\beta\neq\dom\alpha, & \bracks{\text{since } \dom\alpha\neq\im\alpha \text{ and } \beta \text{ is injective}}
	\end{align*}
	which proves that $\parens{\dom\alpha}\beta\neq\dom\alpha$ if and only if $\parens{\im\alpha}\beta\neq\im\alpha$. Since at least one of the inequalities holds, then both inequalities hold and, consequently, we have $\parens{\dom\alpha}\beta=\im\alpha$ and $\parens{\im\alpha}\beta=\dom\alpha$, which proves conditions d) and e).

%	Since, in addition, we have that
%	\begin{displaymath}
%		\dom\gamma_1^{-1}\gamma_1=\im\gamma_1=\parens{\dom\alpha}\beta\neq\parens{\im\alpha}\beta=\im\gamma_2=\dom\gamma_2^{-1}\gamma_2
%	\end{displaymath}
%	(because $\dom\alpha\neq\im\alpha$ and $\beta$ is injective), then we can conclude that $\gamma_1^{-1}\gamma_1\neq\gamma_2^{-1}\gamma_2$ and, consequently, that $\gamma_1^{-1}\gamma_1=\alpha^{-1}\alpha$ and $\gamma_2^{-1}\gamma_2=\alpha\alpha^{-1}$.
	
%	Therefore
%	\begin{gather*}
%		\parens{\dom\alpha}\beta=\im\gamma_1=\dom\gamma_1^{-1}\gamma_1=\dom\alpha^{-1}\alpha=\im\alpha\\
%		\shortintertext{and}
%		\parens{\im\alpha}\beta=\im\gamma_2=\dom\gamma_2^{-1}\gamma_2=\dom\alpha\alpha^{-1}=\dom\alpha,
%	\end{gather*}
%	which proves conditions d) and e).
	
	Finally, we verify that condition f) also holds. Let $\beta\in A$. We have
	\begin{align*}
		\im\beta\alpha &=\parens{\im\beta\cap\dom\alpha}\alpha\\
		&=\parens{\dom\beta\cap\dom\alpha}\alpha &\bracks{\text{by a)}}\\
		&=\parens{\dom\alpha}\alpha &\bracks{\text{since, by c), } \dom\alpha\subseteq\dom\beta}\\
		&=\im\alpha\\
		&\neq\dom\alpha\\
		&=\parens{\im\alpha}\beta &\bracks{\text{by e)}}\\
		&=\parens{\im\alpha\cap\dom\beta}\beta &\bracks{\text{since, by c), } \im\alpha\subseteq\dom\beta}\\
		&=\im\alpha\beta
	\end{align*}
	and
	\begin{align*}
		\im\beta\parens{\alpha\alpha^{-1}} &=\parens{\im\beta\cap\dom\alpha\alpha^{-1}}\alpha\alpha^{-1}\\
		&=\parens{\dom\beta\cap\dom\alpha\alpha^{-1}}\alpha\alpha^{-1} \kern-5mm &\bracks{\text{by a)}}\\
		&=\parens{\dom\alpha\alpha^{-1}}\alpha\alpha^{-1} &\bracks{\text{since, by c), } \dom\alpha\alpha^{-1}=\dom\alpha\subseteq\dom\beta}\\
		&=\im\alpha\alpha^{-1}\\
		&=\dom\alpha\\
		&\neq\im\alpha\\
		&=\parens{\dom\alpha}\beta &\bracks{\text{by d)}}\\
		&=\parens{\im\alpha\alpha^{-1}\cap\dom\beta}\beta &\bracks{\text{since, by c), } \im\alpha\alpha^{-1}=\dom\alpha\subseteq\dom\beta}\\
		&=\im\parens{\alpha\alpha^{-1}}\beta,
	\end{align*}
	which implies that $\beta\alpha\neq\alpha\beta$ and $\beta\parens{\alpha\alpha^{-1}}\neq\parens{\alpha\alpha^{-1}}\beta$. We can verify in an identical way that $\beta\alpha^{-1}\neq\alpha^{-1}\beta$ and $\beta\parens{\alpha^{-1}\alpha}\neq\parens{\alpha^{-1}\alpha}\beta$.

	Next, we establish that $M$ is a Clifford semigroup.
	
	Let $\beta,\beta'\in A$. We have that
	\begin{align*}
		\abs{\dom\beta\beta'}&=\abs{\parens{\im\beta\cap\dom\beta'}\beta^{-1}} \kern -7mm\\
		&= \abs{\im\beta\cap\dom\beta'} &\bracks{\text{since } \beta \text{ is injective}}\\
		&= \abs{\dom\beta\cap\dom\beta'} &\bracks{\text{by a)}}\\
		&\geqslant \abs{\dom\alpha\cup\im\alpha} &\bracks{\text{since, by c), } \dom\alpha\cup\im\alpha\subseteq\dom\beta\cap\dom\beta'}\\
		&>\abs{\dom\alpha}. &\bracks{\text{since } \dom\alpha\neq\im\alpha}
	\end{align*}
	Since $\abs{\dom\alpha}=\abs{\dom\alpha^{-1}}=\abs{\dom\alpha\alpha^{-1}}=\abs{\dom\alpha^{-1}\alpha}$, then we can conclude that $\beta\beta'\in S\setminus\set{\alpha,\alpha^{-1},\alpha\alpha^{-1},\alpha^{-1}\alpha}=M$.
	
	Let $\beta\in A$ and $\beta'\in\centre{S}$. Then we have $\beta\parens{\beta\beta'}=\beta\parens{\beta'\beta}=\parens{\beta\beta'}\beta$, which implies, by f), that $\beta'\beta=\beta\beta'\in S\setminus\set{\alpha,\alpha^{-1},\alpha\alpha^{-1},\alpha^{-1}\alpha}=M$.
	
	Since we also have $\centre{S}^2\subseteq\centre{S}\subseteq M$, then we can conclude that $M=A\cup\centre{S}$ is a subsemigroup of $S$. It is easy to see that for all $\beta\in M=S\setminus\set{\alpha,\alpha^{-1},\alpha\alpha^{-1},\alpha^{-1}\alpha}$ we have $\beta^{-1}\in S\setminus\set{\alpha,\alpha^{-1},\alpha\alpha^{-1},\alpha^{-1}\alpha}=M$, which implies that $M$ is regular. In addition, for all $\beta\in A$ we have
	\begin{align*}
		\parens{\dom\alpha}\beta^2 &=\parens{\im\alpha}\beta &\bracks{\text{by d)}}\\
		&=\dom\alpha &\bracks{\text{by e)}}\\
		&\neq \im\alpha\\
		&=\parens{\dom\alpha}\beta, &\bracks{\text{by d)}}
	\end{align*}
	which implies that $\beta^2\neq\beta$ for all $\beta\in A$, that is, $\idemp{M}\cap A=\emptyset$. Consequently, $\idemp{M}\subseteq M\setminus A=\centre{S}\subseteq\centre{M}$ and, by Theorem~\ref{preli: characterization Clifford smg}, $M$ is a Clifford semigroup. 
	
	At last, we ascertain that $\commgraph{S}$ contains a cycle of length $3$. Let $t_1-t_2-\cdots-t_n-t_1$ be a cycle in $\commgraph{S}$. We saw earlier that the subgraph of $\commgraph{S}$ induced by $\set{\alpha,\alpha^{-1},\alpha\alpha^{-1},\alpha^{-1}\alpha}$ contains no cycles. Furthermore, it follows from condition f) that there is no vertex belonging to $A=\parens{S\setminus\centre{S}}\setminus\set{\alpha,\alpha^{-1},\alpha\alpha^{-1},\alpha^{-1}\alpha}$ that is adjacent to a vertex belonging to $\set{\alpha,\alpha^{-1},\alpha\alpha^{-1},\alpha^{-1}\alpha}$. Thus all the cycles of $\commgraph{S}$ (and, in particular, the cycle $t_1-t_2-\cdots-t_n-t_1$) are also cycles of the subgraph of $\commgraph{S}$ induced by $A$. Let $\mathcal{C}$ be that subgraph. Due to the fact that $M\setminus\centre{M}\subseteq A$ we have that $\commgraph{M}$ is a subgraph of $\mathcal{C}$.
	
	\smallskip
	
	\textsc{Sub-case 1:} Assume that $t_1,\ldots,t_n\in M\setminus\centre{M}$. Then $t_1-t_2-\cdots-t_n-t_1$ is a cycle in $\commgraph{M}$. Since $M$ is a Clifford semigroup, then Theorem~\ref{inv smg: girth Clifford smg} guarantees the existence of a cycle of length $3$ in $\commgraph{M}$. Thus $\mathcal{C}$ --- and, consequently, $\commgraph{S}$ --- contains a cycle of length $3$. Therefore $\girth{\commgraph{S}}=3$.
	
	\smallskip
	
	\textsc{Sub-case 2:} Assume that there exists $i\in\Xn$ such that $t_i\in\centre{M}$. We have that there exist distinct $j,k\in\Xn\setminus\set{i}$ such that $t_j$ is adjacent to $t_k$ (since $t_1-t_2-\cdots-t_n-t_1$ is a cycle in $\mathcal{C}$ and $n\geqslant 3$). Then $t_i-t_j-t_k-t_i$ is a cycle (of length $3$) in $\mathcal{C}$ (and in $\commgraph{S}$). Thus $\girth{\commgraph{S}}=3$.
\end{proof}

We observe that it is easy to find groups (and, consequently, Clifford and inverse semigroups) whose commuting graph has cycles (see, for instance, Figure~\ref{inv smg, Figure: commuting graph of A_4}, which shows the commuting graph of the alternating group over $\set{1,2,3,4}$). There are also groups (and Clifford/inverse semigroups) that contain no cycles --- the symmetric group over $\set{1,2,3}$ is one example (see Figure~\ref{inv smg, Figure: commuting graph of S_3}).

In the last theorem of this section we construct, for every even integer $n\geqslant 4$, a completely regular subsemigroup of $\tr{X}$ (for a convenient finite set $X$) whose commuting graph has a unique cycle that has length $n$.

\begin{theorem}\label{com reg: girth}
	For each $n\in\mathbb{N}$ such that $n\geqslant 2$, there exists a (finite non-commutative) completely regular semigroup whose commuting graph has girth equal to $2n$.
\end{theorem}

\begin{proof}
	We divide the proof into two parts, In the first part we will prove the theorem for $n=2$ and, in the second part, we will prove the theorem for $n\geqslant 3$.
	
	\medskip
	
	\textbf{Part 1.} Let $X=\set{1,2,3}$. We consider the transformations
	\begin{equation}\label{comp reg: girth 4 transformations}
		\alpha_1=\begin{pmatrix}
			1&2&3\\
			1&1&1
		\end{pmatrix}, \quad
		\alpha_2=\begin{pmatrix}
			1&2&3\\
			2&2&2
		\end{pmatrix}, \quad
		\beta_1=\begin{pmatrix}
			1&2&3\\
			1&2&1
		\end{pmatrix}, \quad
		\beta_2=\begin{pmatrix}
			1&2&3\\
			1&2&2
		\end{pmatrix}
	\end{equation}
	whose products are indicated in Table~\ref{comp reg: table}. Let $S=\set{\alpha_1,\alpha_2,\beta_1,\beta_2}$. It is clear, by observation of Table~\ref{comp reg: table}, that $S$ is a subsemigroup of idempotents of $\tr{\set{1,2,3}}$. Hence $\set{\alpha_1}$, $\set{\alpha_2}$, $\set{\beta_1}$ and $\set{\beta_2}$ are subgroups of $\tr{\set{1,2,3}}$ and, consequently, by Theorem~\ref{preli: comp reg, every element in subgroup}, we have that $S$ is completely regular. In addition, $S$ is not commutative (for example, we have $\alpha_1\alpha_2\neq\alpha_2\alpha_1$). In Figure~\ref{comp reg: figure, cycle length 4} we have an illustration of $\commgraph{S}$, which is a cycle graph of length $4$. Therefore $\girth{\commgraph{S}}=4$.
	
	\begin{table}[h!bt]
		\centering
		\begin{NiceTabular}{c|cccc}
			& $\alpha_1$ & $\alpha_2$ & $\beta_1$ & $\beta_2$\\
			\midrule
			$\alpha_1$ & $\alpha_1$ & $\alpha_2$ & $\alpha_1$ & $\alpha_1$\\
			$\alpha_2$ & $\alpha_1$ & $\alpha_2$ & $\alpha_2$ & $\alpha_2$ \\
			$\beta_1$ & $\alpha_1$ & $\alpha_2$ & $\beta_1$ & $\beta_1$ \\
			$\beta_2$ & $\alpha_1$ & $\alpha_2$ & $\beta_2$ & $\beta_2$ \\
		\end{NiceTabular}
		\caption{Products of $\alpha_1$, $\alpha_2$, $\beta_1$, $\beta_2$ defined in \eqref{comp reg: girth 4 transformations}.}
		\label{comp reg: table}
	\end{table}

	\begin{figure}[h!bt]
		\centering
		\begin{tikzpicture}
			
			\node[vertex] (a1) at (0,0) {};
			\node[vertex] (b1) at (1.5,0) {};
			\node[vertex] (b2) at (0,-1.5) {};
			\node[vertex] (a2) at (1.5,-1.5) {};
			
			\node[anchor=south east] at (a1) {$\alpha_1$};
			\node[anchor=north west] at (a2) {$\alpha_2$};
			\node[anchor=south west] at (b1) {$\beta_1$};
			\node[anchor=north east] at (b2) {$\beta_2$};
			
			\begin{scope}[edge]
				\draw (a1) -- (b1);
				\draw (b1) -- (a2);
				\draw (a2) -- (b2);
				\draw (b2) -- (a1);
			\end{scope}
			
		\end{tikzpicture}
		\caption{Commuting graph of a completely regular semigroup with girth equal to $4$.}
		\label{comp reg: figure, cycle length 4}
	\end{figure}
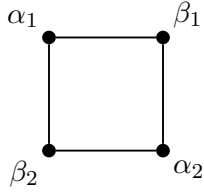
	
	\medskip
	
	\textbf{Part 2.} Let $n\in\mathbb{N}$ be such that $n\geqslant 3$. For each $i\in\set{0,\ldots,n-1}$ let $X_i=\set{x_i^1,\ldots,x_i^{n-1}}$ (the superscripts are additional indices, not exponents). Let $X=\bigcup_{i=0}^{n-1} X_i$. It is clear that $\set{X_i}_{i=0}^{n-1}$ is a partition of $X$.
	
	We are going to construct a subsemigroup of idempotents of $\tr{X}$. In order to do that, we first introduce $n^2$ transformations of $\tr{X}$. For each $i\in\set{0,\ldots,n-1}$ we define $\alpha_i$ as the transformation such that:
	\begin{align*}
		&X_i\alpha_i=\set{x_i^1}\\
		&X_{i+j\bmod{n}}\alpha_i=\set{x_{i+j \bmod{n}}^j}, \text{ } j=1,\ldots,n-1.
	\end{align*}
	For each $i\in\set{0,\ldots,n-1}$ and $j\in\X{n-1}$ we also define $\beta_i^j$ as the transformation such that $\im \beta_i^j=\set{x_i^j}$ (the superscript in $\beta_i^j$ is an additional index, not an exponent). Let
	\begin{displaymath}
		S=\gset{\alpha_i}{i=0,\ldots,n-1} \cup \gset{\beta_i^j}{i=0,\ldots,n-1 \text{ and } j=1,\ldots,n-1}.
	\end{displaymath}
	We aim to prove that $S$ is a (non-commutative) subsemigroup of idempotents of $\tr{X}$.

	First, we are going to confirm that the $n^2$ transformations we defined are pairwise distinct. We note that for all $i,k\in\set{0,\ldots,n-1}$ and $j\in\X{n-1}$ we have that $\beta_i^j$ has rank $1$ and $\alpha_k$ has rank $n\geqslant 3$. Hence $\beta_i^j\neq\alpha_k$ for all $i,k\in\set{0,\ldots,n-1}$ and $j\in\X{n-1}$. Moreover, for all $i,k\in\set{0,\ldots,n-1}$ and $j,m\in\X{n-1}$ such that $\parens{i,j}\neq\parens{k,m}$ we have that $\im\beta_i^j=\set{x_i^j}\neq \set{x_k^m}=\im\beta_k^m$, which implies that $\beta_i^j\neq\beta_k^m$ for all $i,k\in\set{0,\ldots,n-1}$ and $j,m\in\X{n-1}$ such that $\parens{i,j}\neq\parens{k,m}$. The only thing left to check is that $\alpha_i\neq\alpha_k$, for all distinct $i,k\in\set{0,\ldots,n-1}$. Let $i,k\in\set{0,\ldots,n-1}$ be such that $i<k$. We consider two cases.
	
	\smallskip
	
	\textit{Case 1:} Assume that $k=i+1$. Then we have
	\begin{align*}
		X_i\alpha_k &=X_{k-1}\alpha_k &\bracks{\text{since } k=i+1}\\
		&=X_{k+\parens{n-1}\bmod{n}}\alpha_k &\bracks{\text{since } k-1=i\in\set{0,\ldots,n-1}}\\
		&=\set{x_{k+\parens{n-1}\bmod{n}}^{n-1}}\\
		&=\set{x_{i+n\bmod{n}}^{n-1}} &\bracks{\text{since } k=i+1}\\
		&=\set{x_{i}^{n-1}} &\bracks{\text{since } i\in\set{0,\ldots,n-1}}\\
		&\neq\set{x_i^{1}} &\bracks{\text{since } n\geqslant 3}\\
		&=X_i\alpha_i,
	\end{align*}
	which implies that $\alpha_i\neq\alpha_k$.
	
	\smallskip
	
	\textit{Case 2:} Assume that $k>i+1$. Then we have
	\begin{align*}
		X_k\alpha_i &=X_{k\bmod{n}}\alpha_i &\bracks{\text{since } k\in\set{0,\ldots,n-1}}\\
		&=X_{i+\parens{k-i}\bmod{n}}\alpha_i\\
		&=\set{x_{i+\parens{k-i}\bmod{n}}^{k-i}} &\bracks{\text{since } k-i\in\set{0,\ldots,n-1}}\\
		&=\set{x_{k}^{k-i}}\\
		&\neq\set{x_k^1} &\bracks{\text{since } k-i>i+1-i=1}\\
		&=X_k\alpha_k,
	\end{align*}
	which implies that $\alpha_i\neq\alpha_k$.
	
	\smallskip
	
	Now we are going to check that $S$ is a semigroup. With this goal in mind, we are going to prove that for all $i,k\in\set{0,\ldots,n-1}$ and $j,m\in\X{n-1}$ we have:
	\begin{align}
		\alpha_i\alpha_k&=\alpha_k;
		\label{comp reg: prod alphas}\\
		\beta_i^j\beta_k^m&=\beta_k^m;
		\label{comp reg: prod betas}\\
		\alpha_i\beta_k^m&=\beta_k^m;
		\label{comp reg: prod alpha beta}\\
		\beta_k^m\alpha_i&=\begin{cases}
			\beta_k^1 &\text{if } l=0,\\
			\beta_k^l &\text{if } l\in\X{n-1},
		\end{cases}
		\label{comp reg: prod beta alpha}\\
		& \kern 25mm \text{where } l\in\set{0,\ldots,n-1} \text{ is such that } k=i+l \bmod{n}. \kern -25mm
		\nonumber
	\end{align}

	Let $i,k\in\set{0,\ldots,n-1}$ and $j,m\in\X{n-1}$ and let $x\in X$.
	
	Since $\im\beta_k^m=\set{x_k^m}$, then we have $x\beta_i^j\beta_k^m=x_k^m=x\beta_k^m$ and $x\alpha_i\beta_k^m=x_k^m=\beta_k^m$. Hence $\beta_i^j\beta_k^m=\beta_k^m$ and $\alpha_i\beta_k^m=\beta_k^m$, which proves \eqref{comp reg: prod betas} and \eqref{comp reg: prod alpha beta}. 
	
	Now we are going to prove \eqref{comp reg: prod alphas}. Let $t\in\set{0,\ldots,n-1}$ be such that $x\in X_t$. We have that $x\alpha_i=x_t^s\in X_t$ for some $s\in\X{n-1}$. Let $l\in\set{0,\ldots,n-1}$ be such that $t=k+l \bmod{n}$. Then  
	\begin{align*}
		x\alpha_i\alpha_k&=x_t^s\alpha_k &\bracks{\text{since } x\alpha_i=x_t^s}\\
		&=x_{k+l \bmod{n}}^s\alpha_k &\bracks{\text{since } t=k+l \bmod{n}}\\
		&=\begin{cases}
			x_k^1 &\text{if } l=0,\\
			x_{k+l \bmod{n}}^l &\text{if } l\in\X{n-1}
		\end{cases}\\
		&=x\alpha_k. &\bracks{\text{since } x\in X_t \text{ and } t=k+l \bmod{n}}
	\end{align*}
	Thus $\alpha_i\alpha_k=\alpha_k$.
	
	Finally, we are going to prove \eqref{comp reg: prod beta alpha}. Let $l\in\set{0,\ldots,n-1}$ be such that $k=i+l \bmod{n}$. We have that
	\begin{align*}
		x\beta_k^m\alpha_i &=x_k^m\alpha_i &\bracks{\text{since} \im\beta_k^m=\set{x_k^m}}\\
		&=x_{i+l \bmod{n}}^m\alpha_i &\bracks{\text{since } k=i+l \bmod{n}}\\
		&=\begin{cases}
			x_i^1 &\text{if } l=0,\\
			x_{i+l \bmod{n}}^l &\text{if } l\in\X{n-1}
		\end{cases}\\
		&=\begin{cases}
			x\beta_i^1 &\text{if } l=0,\\
			x\beta_{i+l \bmod{n}}^l &\text{if } l\in\X{n-1}
		\end{cases} &\smash{\begin{minipage}[c]{50mm}\raggedleft [since $\im\beta_t^s=\set{x_t^s}$, $t=0,\ldots,n-1$, $s=0,\ldots,n-1$]\end{minipage}}\\	
		&=\begin{cases}
			x\beta_k^1 &\text{if } l=0,\\
			x\beta_k^l &\text{if } l\in\X{n-1}
		\end{cases} &\bracks{\text{since } k=i+l \bmod{n}}
	\end{align*}
	which proves \eqref{comp reg: prod beta alpha}.
	
	It is clear that \eqref{comp reg: prod alphas}--\eqref{comp reg: prod beta alpha} allow us to conclude that $S$ is closed under composition; that is, $S$ is a subemigroup of $\tr{X}$. Furthermore, it follows from \eqref{comp reg: prod alphas} that $\alpha_i\alpha_i=\alpha_i$ for all $i\in\set{0,\ldots,n-1}$, and it follows from \eqref{comp reg: prod betas} that $\beta_i^j\beta_i^j=\beta_i^j$ for all $i\in\set{0,\ldots,n-1}$ and $j\in\X{n-1}$. Therefore $S$ is a semigroup of idempotents. This implies that the singletons $\set{\alpha_i}$, where $i\in\set{0,\ldots,n-1}$, and $\set{\beta_i^j}$, where $i\in\set{0,\ldots,n-1}$ and $j\in\X{n-1}$, are subgroups of $S$. Thus, by Theorem~\ref{preli: comp reg, every element in subgroup}, we have that $S$ is a completely regular semigroup.
	
	Our next objective is to determine which transformations of $S$ commute. More precisely, we will demonstrate that for all $i,k\in\set{0,\ldots,n-1}$ and $j,m\in\X{n-1}$ we have
	\begin{align*}
		\alpha_i\alpha_k=\alpha_k\alpha_i &\iff i=k,\\
		%\label{com reg: comm alpha}\\
		\beta_i^j\beta_k^m=\beta_k^m\beta_i^j &\iff i=k \text{ and } j=m,\\
		%\label{com reg: comm beta}\\
		\alpha_i\beta_k^1=\beta_k^1\alpha_i &\iff k=i \text{ or } k=i+1\bmod{n},\\
		%\label{com reg: comm alpha beta 1}\\
		\alpha_i\beta_k^m=\beta_k^m\alpha_i &\iff k=i+m\bmod{n}.
		%\label{com reg: comm alpha beta}
	\end{align*}
	
	Let $i,k\in\set{0,\ldots,n-1}$, $j\in\X{n-1}$ and $m\in\set{2,\ldots,n-1}$. Let $l\in\set{0,\ldots,n-1}$ be such that $k=i+l\bmod{n}$.
	
	By \eqref{comp reg: prod alphas} we have
	\begin{displaymath}
		\alpha_i\alpha_k=\alpha_k\alpha_i\iff\alpha_k=\alpha_i\iff k=i.
	\end{displaymath}
	
	By \eqref{comp reg: prod betas} we have
	\begin{displaymath}
		\beta_i^j\beta_k^m=\beta_k^m\beta_i^j \iff \beta_k^m=\beta_i^j \iff k=i \text{ and } m=j.
	\end{displaymath}
	
	By \eqref{comp reg: prod alpha beta} and \eqref{comp reg: prod beta alpha} we have
	\begin{align*}
		\alpha_i\beta_k^1=\beta_k^1\alpha_i &\iff \beta_k^1=\beta_k^1\alpha_i \\
		&\iff l=0 \text{ or } \parens{l\geqslant 1 \text{ and } \beta_k^1=\beta_k^l}\\
		&\iff l=0 \text{ or } l=1 \\
		&\iff k=i \text{ or } k=i+1\bmod{n}.\\
		\shortintertext{and}
		\alpha_i\beta_k^m=\beta_k^m\alpha_i &\iff \beta_k^m=\beta_k^m\alpha_i \\
		&\iff l\geqslant 1 \text{ and } \beta_k^m=\beta_k^l &\bracks{\text{since } m>1}\\
		&\iff m=l &\bracks{\text{since } m>1}\\
		&\iff k=i+m\bmod{n}. &\bracks{\text{since } k=i+l\bmod{n}}
	\end{align*}
	
	%	We have that:
	%	\begin{enumerate}
		%		\item If $i\in\set{0,\ldots,n-1}$, then \eqref{com reg: comm alpha} implies that vertex $\alpha_i$ is not adjacent to any of the vertices $\alpha_k$, where $i\in\set{0,\ldots,n-1}\setminus\set{k}$.
		%		
		%		\item If $i\in\set{0,\ldots,n-1}$ and $j\in\X{n-1}$, then \eqref{com reg: comm beta} implies that vertex $\beta_i^j$ is not adjacent to any of the vertices $\beta_k^m$, where $k\in\set{0,\ldots,n-1}$ and $m\in\X{n-1}$ are such that $\parens{i,j}\neq\parens{k,m}$.
		%		
		%		\item If $i\in\set{0,\ldots,n-1}$, then \eqref{com reg: comm alpha beta 1} implies that, among the vertices $\beta_0^1,\beta_1^1,\ldots,\beta_{n-1}^1$, there are only two that are adjacent to vertex $\alpha_i$, namely, $\beta_i^1$ and $\beta_{i+1\bmod{n}}^1$.
		%		
		%		\item If $k\in\set{0,\ldots,n-1}$, then \eqref{com reg: comm alpha beta 1} implies that vertex $\beta_k^1$ is only adjacent to two vertices of $\commgraph{S}$, namely, $\alpha_k$ and $\alpha_{k+1\bmod{n}}$.
		%		
		%		\item If $k\in\set{0,\ldots,n-1}$ and $m\in\set{2,\ldots,n-1}$, then \eqref{com reg: comm alpha beta} implies that vertex $\beta_k^m$ is adjacent to only one vertex of $\commgraph{S}$, namely, $\alpha_i$, where $i\in\set{0,\ldots,n-1}$ is such that $k=i+l\bmod{n}$.
		%	\end{enumerate}
	
	With the previous four equivalences we can construct $\commgraph{S}$, which is illustrated in Figure~\ref{comp reg: figure, cycle length 2n}. Moreover, we see that the unique cycle of $\commgraph{S}$ is
	\begin{displaymath}
		\alpha_0-\beta_1^1-\alpha_1-\beta_2^1-\alpha_2-\cdots-\alpha_{i-1}-\beta_i^1-\alpha_i-\beta_{i+1}^1-\cdots-\beta_{n-1}^1-\alpha_{n-1}-\beta_0^1-\alpha_0.
	\end{displaymath}
	Therefore $\girth{\commgraph{S}}=2n$, which concludes the proof.
	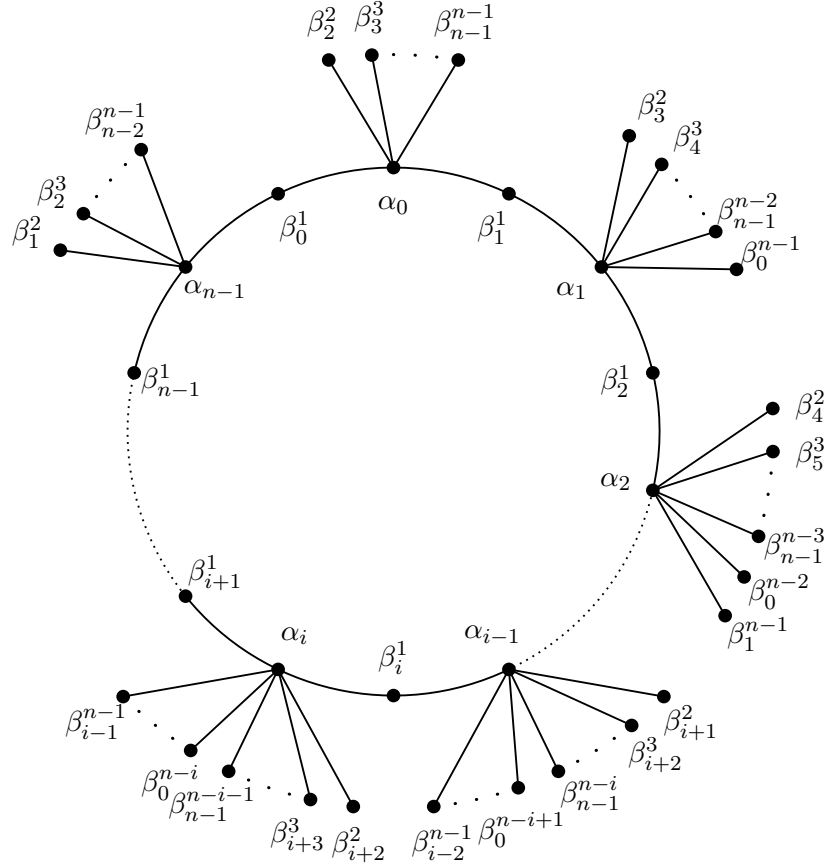
\begin{figure}[h!bt]
		\centering
		\begin{tikzpicture}
			
			\foreach \x in {-4,-3,-2,-1,0,1,2,3,5,6,7,8} {
				\node[vertex] at (90+\x*360/14:3.5cm) {};
			}
			
			\draw (90+2*360/14:3cm) node {$\alpha_{n-1}$};
			\draw (90:3cm) node {$\alpha_0$};
			\draw (90-2*360/14:3cm) node {$\alpha_1$};
			\draw (90-4*360/14:3cm) node {$\alpha_2$};
			\draw (90-6*360/14:3cm) node {$\alpha_{i-1}$};
			\draw (90-8*360/14:3cm) node {$\alpha_i$};
			
			\draw (90+3*360/14:3cm) node {$\beta_{n-1}^1$};
			\draw (90+360/14:3cm) node {$\beta_0^1$};
			\draw (90-1*360/14:3cm) node {$\beta_1^1$};
			\draw (90-3*360/14:3cm) node {$\beta_2^1$};
			\draw (90-7*360/14:3cm) node {$\beta_i^1$};
			\draw (90-9*360/14:3cm) node {$\beta_{i+1}^1$};
			
			\draw[edge] (90-4*360/14:3.5cm) arc (90-4*360/14:90+3*360/14:3.5cm);
			\draw[edge] (90-6*360/14:3.5cm) arc (90-6*360/14:90-9*360/14:3.5cm);
			\draw[edge, dotted] (90-4*360/14:3.5cm) arc (90-4*360/14:90-6*360/14:3.5cm);
			\draw[edge, dotted] (90-9*360/14:3.5cm) arc (90-9*360/14:90-11*360/14:3.5cm);
			
			%%%%%%%%%%%%%%%%%%%%%%%%%%%%%%%%%%%%%%%%
			
			%\alpha_0
			
			\foreach \x in {-1.5,0.5,1.5} {
				\node[vertex] at (90+\x*360/55:5cm) {};
			}
			
			\draw (90+1.5*360/55:5.5cm) node {$\beta_2^2$};
			\draw (90+0.5*360/55:5.5cm) node {$\beta_3^3$};
			\draw (90-1.5*360/55:5.5cm) node {$\beta_{n-1}^{n-1}$};
			
			\draw[edge] (90:3.5cm) -- (90+1.5*360/55:5cm);
			\draw[edge] (90:3.5cm) -- (90+0.5*360/55:5cm);
			\draw[edge] (90:3.5cm) -- (90-1.5*360/55:5cm);
			
			\draw[dash pattern=on 0pt off 8pt,line cap=round, very thick] (90+0.5*360/55:5cm) arc (90+0.5*360/55:90-1.5*360/55:5cm);
			
			%%%%%%%%%%%%%%%%%%%%%%%%%%%%%%%%%%%%%%%%
			
			%\alpha_1
			
			\foreach \x in {-2,-1,1,2} {
				\node[vertex] at (90-2*360/14+\x*360/55:5cm) {};
			}
			
			\draw (90-2*360/14+2*360/55:5.5cm) node {$\beta_3^2$};
			\draw (90-2*360/14+1*360/55:5.5cm) node {$\beta_4^3$};
			\draw (90-2*360/14-1*360/55:5.5cm) node {$\beta_{n-1}^{n-2}$};
			\draw (90-2*360/14-2*360/55:5.5cm) node {$\beta_0^{n-1}$};
			
			\draw[edge] (90-2*360/14:3.5cm) -- (90-2*360/14+2*360/55:5cm);
			\draw[edge] (90-2*360/14:3.5cm) -- (90-2*360/14+1*360/55:5cm);
			\draw[edge] (90-2*360/14:3.5cm) -- (90-2*360/14-1*360/55:5cm);
			\draw[edge] (90-2*360/14:3.5cm) -- (90-2*360/14-2*360/55:5cm);
			
			\draw[dash pattern=on 0pt off 8pt,line cap=round, very thick] (90-2*360/14+1*360/55:5cm) arc (90-2*360/14+1*360/55:90-2*360/14-1*360/55:5cm);
			
			%%%%%%%%%%%%%%%%%%%%%%%%%%%%%%%%%%%%%%%%
			
			%\alpha_2
			
			\foreach \x in {-2.5,-1.5,-0.5,1.5,2.5} {
				\node[vertex] at (90-4*360/14+\x*360/55:5cm) {};
			}
			
			\draw (90-4*360/14+2.5*360/55:5.5cm) node {$\beta_4^2$};
			\draw (90-4*360/14+1.5*360/55:5.5cm) node {$\beta_5^3$};
			\draw (90-4*360/14-0.5*360/55:5.5cm) node {$\beta_{n-1}^{n-3}$};
			\draw (90-4*360/14-1.5*360/55:5.5cm) node {$\beta_0^{n-2}$};
			\draw (90-4*360/14-2.5*360/55:5.5cm) node {$\beta_1^{n-1}$};
			
			\draw[edge] (90-4*360/14:3.5cm) -- (90-4*360/14+2.5*360/55:5cm);
			\draw[edge] (90-4*360/14:3.5cm) -- (90-4*360/14+1.5*360/55:5cm);
			\draw[edge] (90-4*360/14:3.5cm) -- (90-4*360/14-0.5*360/55:5cm);
			\draw[edge] (90-4*360/14:3.5cm) -- (90-4*360/14-1.5*360/55:5cm);
			\draw[edge] (90-4*360/14:3.5cm) -- (90-4*360/14-2.5*360/55:5cm);
			
			\draw[dash pattern=on 0pt off 8pt,line cap=round, very thick] (90-4*360/14+1.5*360/55:5cm) arc (90-4*360/14+1.5*360/55:90-4*360/14-0.5*360/55:5cm);
			
			%%%%%%%%%%%%%%%%%%%%%%%%%%%%%%%%%%%%%%%%
			
			%\alpha_{i-1}
			
			\foreach \x in {-3,-1,0,2,3} {
				\node[vertex] at (90-6*360/14+\x*360/55:5cm) {};
			}
			
			\draw (90-6*360/14+3*360/55:5.5cm) node {$\beta_{i+1}^2$};
			\draw (90-6*360/14+2*360/55:5.5cm) node {$\beta_{i+2}^3$};
			\draw (90-6*360/14+0.4*360/55:5.46cm) node {$\beta_{n-1}^{n-i}$};
			\draw (90-6*360/14-1.22*360/55:5.54cm) node {$\beta_0^{n-i+1}$};
			\draw (90-6*360/14-2.9*360/55:5.5cm) node {$\beta_{i-2}^{n-1}$};
			
			\draw[edge] (90-6*360/14:3.5cm) -- (90-6*360/14+3*360/55:5cm);
			\draw[edge] (90-6*360/14:3.5cm) -- (90-6*360/14+2*360/55:5cm);
			\draw[edge] (90-6*360/14:3.5cm) -- (90-6*360/14:5cm);
			\draw[edge] (90-6*360/14:3.5cm) -- (90-6*360/14-1*360/55:5cm);
			\draw[edge] (90-6*360/14:3.5cm) -- (90-6*360/14-3*360/55:5cm);
			
			\draw[dash pattern=on 0pt off 8pt,line cap=round, very thick] (90-6*360/14+2*360/55:5cm) arc (90-6*360/14+2*360/55:90-6*360/14:5cm);
			\draw[dash pattern=on 0pt off 8pt,line cap=round, very thick] (90-6*360/14-3*360/55:5cm) arc (90-6*360/14-3*360/55:90-6*360/14-1*360/55:5cm);

			%%%%%%%%%%%%%%%%%%%%%%%%%%%%%%%%%%%%%%%%
			
			%\alpha_i
			
			\foreach \x in {-3,-1,0,2,3} {
				\node[vertex] at (90-8*360/14+\x*360/55:5cm) {};
			}
			
			\draw (90-8*360/14+3.2*360/55:5.5cm) node {$\beta_{i+2}^2$};
			\draw (90-8*360/14+1.9*360/55:5.5cm) node {$\beta_{i+3}^3$};
			\draw (90-8*360/14:5.5cm) node {$\beta_{n-1}^{n-i-1}$};
			\draw (90-8*360/14-1*360/55:5.5cm) node {$\beta_0^{n-i}$};
			\draw (90-8*360/14-3*360/55:5.5cm) node {$\beta_{i-1}^{n-1}$};
			
			\draw[edge] (90-8*360/14:3.5cm) -- (90-8*360/14+3*360/55:5cm);
			\draw[edge] (90-8*360/14:3.5cm) -- (90-8*360/14+2*360/55:5cm);
			\draw[edge] (90-8*360/14:3.5cm) -- (90-8*360/14:5cm);
			\draw[edge] (90-8*360/14:3.5cm) -- (90-8*360/14-1*360/55:5cm);
			\draw[edge] (90-8*360/14:3.5cm) -- (90-8*360/14-3*360/55:5cm);
			
			\draw[dash pattern=on 0pt off 8pt,line cap=round, very thick] (90-8*360/14+2*360/55:5cm) arc (90-8*360/14+2*360/55:90-8*360/14:5cm);
			\draw[dash pattern=on 0pt off 8pt,line cap=round, very thick] (90-8*360/14-3*360/55:5cm) arc (90-8*360/14-3*360/55:90-8*360/14-1*360/55:5cm);
			
			%%%%%%%%%%%%%%%%%%%%%%%%%%%%%%%%%%%%%%%%
			
			%\alpha_{n-1}
			
			\foreach \x in {-1.5,0.5,1.5} {
				\node[vertex] at (90+2*360/14+\x*360/55:5cm) {};
			}
			
			\draw (90+2*360/14+1.5*360/55:5.5cm) node {$\beta_1^2$};
			\draw (90+2*360/14+0.5*360/55:5.5cm) node {$\beta_2^3$};
			\draw (90+2*360/14-1.5*360/55:5.5cm) node {$\beta_{n-2}^{n-1}$};
			
			\draw[edge] (90+2*360/14:3.5cm) -- (90+2*360/14+1.5*360/55:5cm);
			\draw[edge] (90+2*360/14:3.5cm) -- (90+2*360/14+0.5*360/55:5cm);
			\draw[edge] (90+2*360/14:3.5cm) -- (90+2*360/14-1.5*360/55:5cm);
			
			\draw[dash pattern=on 0pt off 8pt,line cap=round, very thick] (90+2*360/14+0.5*360/55:5cm) arc (90+2*360/14+0.5*360/55:90+2*360/14-1.5*360/55:5cm);
			
		\end{tikzpicture}
		\caption{Commuting graph of a completely regular semigroup whose girth equal to $2n$, where $n\geqslant 3$.}
		\label{comp reg: figure, cycle length 2n}
	\end{figure}
\end{proof}

\section{Possible values for the clique number of commuting graphs}\label{sec: clique number inv smg}

The purpose of this section is to investigate the possible values for the clique numbers of commuting graphs of groups, Clifford semigroups, inverse semigroups and completely regular semigroups. In \cite[Corollary 4.1]{Completely_simple_semigroups_paper} the present author showed that the set of possible values for the clique number of the commuting graph of a completely simple semigroup is $\mathbb{N}$. Since completely simple semigroups are completely regular, then $\mathbb{N}$ is also the set of possible values for the clique number of the commuting graph of a completely regular semigroup. This means we only need to consider the classes of groups, Clifford semigroups and inverse semigroups.

We require two results from \cite{Semigroup_constructions_paper}: the first one concerns the clique number of the commuting graph of a zero-union of semigroups, and the second one concerns the clique number of the commuting graph of a direct product of semigroups.

Let $n\in\mathbb{N}\setminus\set{1}$. We recall that the \defterm{zero-union} of $n$ disjoint semigroups $S_1,\ldots,S_n$ is the set $\set{0}\cup\bigcup_{i=1}^n S_i$, where $0$ is a new element, and where the product of any two elements $x$ and $y$ is equal to the element $xy\in S_i$, if $x,y\in S_i$ for some $i\in\Xn$; and $0$ for the remaining cases. The \defterm{direct product} of $n$ semigroups $S_1,\ldots,S_n$ is the set $\prod_{i=1}^n S_i$ with componentwise multiplication: $\parens{i}\parens{st}=\parens{i}s\parens{i}t$ for all $i\in\Xn$ and $s,t\in \prod_{i=1}^n S_i$.

\begin{corollary}\cite[Theorem 3.2 and Corollary 3.4]{Semigroup_constructions_paper}\label{0Union: clique number}
	Let $n\in\mathbb{N}\setminus\set{1}$ and let $S$ be the zero-union of the finite semigroups $S_1,\ldots,S_n$. Let $\NC=\{i\in\Xn: S_i \text{ is not}\allowbreak \text{commutative}\}$ and suppose that $\NC\neq\emptyset$. We have that $\cliquenumber{\commgraph{S}}=\sum_{i\in\NC}\cliquenumber{\commgraph{S_i}}$.
\end{corollary}

\begin{theorem}\cite[Theorem 4.6]{Semigroup_constructions_paper}\label{direct prod: clique number}
	Let $n\in\mathbb{N}\setminus\set{1}$ and let $S$ be the direct product of the finite semigroups $S_1,\ldots,S_n$. Let $C=\gset{i\in\Xn}{S_i \text{ is commutative}}$ and let $\NC=\gset{i\in\Xn}{S_i \text{ is not commutative}}$ and suppose that $\NC\neq\emptyset$. We have that
	\begin{displaymath}
		\cliquenumber{\commgraph{S}}=\parens[\bigg]{\,\prod_{i\in C}\abs{S_i}}\parens[\bigg]{\,\prod_{i\in \NC}{\parens[\big]{\abs{\centre{S_i}}+\cliquenumber{\commgraph{S_i}}}}} -\prod_{i=1}^n{\abs{\centre{S_i}}}.
	\end{displaymath}
\end{theorem}

Now we are ready to verify that all even numbers are possible values for the clique number of the commuting graph of a group.

\begin{theorem}\label{inv smg: clique number groups}
	For each $n\in\mathbb{N}$, there is a (finite non-abelian) group whose commuting graph has clique number equal to $2n$.
\end{theorem}

\begin{proof}
	Let $n\in\mathbb{N}$. We consider two groups: $\Sym{3}$ --- the symmetric group over $\set{1,2,3}$ --- and $C_n$ --- the cyclic group of order $n$. Then $\Sym{3}\times C_n$ is also a group. Furthermore, since $\Sym{3}$ is not commutative, then $\Sym{3}\times C_n$ is not commutative (see \cite[Proposition 4.2]{Semigroup_constructions_paper}). We have $\centre{C_n}=C_n$ (because $C_n$ is abelian) and it is easy to see that the only central element of $\Sym{3}$ is the identity. Moreover, by examination of Figure~\ref{inv smg, Figure: commuting graph of S_3} we can see that $\cliquenumber{\commgraph{\Sym{3}}}=2$. Then, by Theorem~\ref{direct prod: clique number}, we have
	\begin{align*}
		\cliquenumber{\commgraph{\Sym{3}\times C_n}}& =\parens[\big]{\abs{\centre{\Sym{3}}}+\cliquenumber{\commgraph{\Sym{3}}}}\cdot\abs{C_n}-\parens[\big]{\abs{\centre{\Sym{3}}}\cdot\abs{\centre{C_n}}}\\
		& =3n - n\\
		& =2n.\qedhere
	\end{align*}
\end{proof}

\begin{figure}[hbt]
	\begin{center}
		\begin{tikzpicture}
			
			\node[vertex] (231) at (0,0) {};
			\node[vertex] (312) at (0,-1) {};
			\node[vertex] (132) at (2,-0.5) {};
			\node[vertex] (213) at (4,-0.5) {};
			\node[vertex] (321) at (6,-0.5) {};

			\node[anchor=south east] at (231) {$\begin{pmatrix}1&2&3\\ 2&3&1\end{pmatrix}$};
			\node[anchor=north east] at (312) {$\begin{pmatrix}1&2&3 \\ 3&1&2\end{pmatrix}$};
			\node[anchor=south] at (132) {$\begin{pmatrix}1&2&3 \\ 1&3&2\end{pmatrix}$};
			\node[anchor=south] at (213) {$\begin{pmatrix}1&2&3 \\ 2&1&3\end{pmatrix}$};
			\node[anchor=south] at (321) {$\begin{pmatrix}1&2&3 \\ 3&2&1\end{pmatrix}$};
			
			\draw[edge] (231)--(312);
			
		\end{tikzpicture}
	\end{center}
	\caption{Commuting graph of the symmetric group $\Sym{3}$ on $\set{1,2,3}$.}
	\label{inv smg, Figure: commuting graph of S_3}
\end{figure}
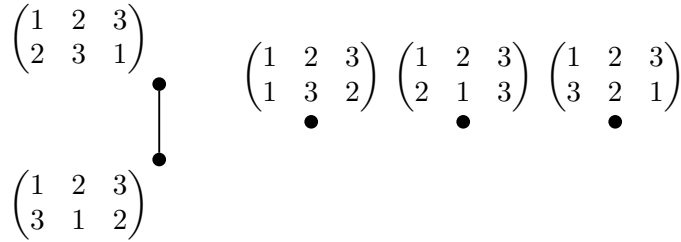

We note that it is not true that the clique number of the commuting graph of a group is always an even number. In fact, there are several groups whose commuting graph has clique number equal to an odd number --- for instance, it follows from {\cite[Theorem 1]{Symmetric_group}} and \cite[Theorem 1.1]{Alternating_group} that for all $n\in\mathbb{N}$, the clique numbers of the commuting graphs of the symmetric and alternating groups $\Sym{3n+1}$ and $\Alt{3n+1}$ on $\X{3n+1}$ is $4\cdot3^{k-1}-1$, which is odd.

Now we establish that each positive integer greater than $1$ is a possible value for the clique number of the commuting graph of a Clifford semigroup. Since Clifford semigroups are simultaneously inverse and completely regular, then the result also holds for these semigroups. In order to prove the desired result, we need the following lemma.

\begin{proposition}\label{inv smg: zero-union Clifford smg is Clifford smg}
	Let $n\in\mathbb{N}$ and let $S_1,\ldots,S_n$ be Clifford semigroups. Then the zero-union of $S_1,\ldots,S_n$ is also a Clifford semigroup.
\end{proposition}

\begin{proof}
	Let $S=\set{0}\cup\bigcup_{i=1}^n S_i$ be the zero-union of $S_1,\ldots,S_n$. Since $S_1,\ldots,S_n$ are Clifford semigroups, then they are regular and their idempotents are central.
	
	We have $\idemp{S}=\set{0}\cup\bigcup_{i=1}^n\idemp{S_i}$. We are going to see that $\idemp{S}\subseteq\centre{S}$. Let $x\in\idemp{S}$.
	
	\smallskip
	
	\textit{Case 1:} Assume that $x=0$. Then $xy=0y=0=y0=yx$ for all $y\in S$, which implies that $x\in\centre{S}$.
	
	\smallskip
	
	\textit{Case 2:} Assume that $x\in S_i$ for some $i\in\Xn$. Due to the fact that $x$ is an idempotent of $S_i$ and all idempotents of $S_i$ are central, then $xy=yx$ for all $y\in S_i$. Furthermore, $xy=0=yx$ for all $j\in\Xn\setminus\set{i}$ and $y\in S_j$. Since we also have $x0=0=0x$, then $x\in\centre{S}$.
	
	\smallskip
	
	We just proved that all idempotents of $S$ are central. Additionally, since $0$ is regular and $S_1,\ldots,S_n$ are regular, then $S$ is regular. Therefore $S$ is a Clifford semigroup.
\end{proof}

\begin{theorem}\label{inv smg: clique number Clifford}
	For each $n\in\mathbb{N}$ such that $n\geqslant 2$, there is a (finite non-commutative) Clifford semigroup whose commuting graph has clique number equal to $n$.
\end{theorem}

\begin{proof}
	Assume that $n$ is even. Then, by Theorem~\ref{inv smg: clique number groups}, there exists a finite non-abelian group $G$ such that $\cliquenumber{\commgraph{G}}=n$. The result follows from the fact that groups are Clifford semigroups.
	
	Assume that $n=3$. The alternating group $\Alt{4}$ over $\set{1,2,3,4}$ is a Clifford semigroup (because it is a group). Moreover, it is easy to verify in Figure~\ref{inv smg, Figure: commuting graph of A_4} that $\cliquenumber{\commgraph{\Alt{4}}}=3=n$.
	
	Now assume that $n$ is odd and that $n\geqslant 5$. Let $k\in\mathbb{N}$ be such that $n=2k+1$. Then $k\geqslant 2$. Let $S$ be the zero-union of $\Alt{4}$ with $k-1$ copies of $\Sym{3}$. Since $\Alt{4}$ and $\Sym{3}$ are groups (and, consequently, Clifford semigroups), then Proposition~\ref{inv smg: zero-union Clifford smg is Clifford smg} guarantees that $S$ is a Clifford semigroup. Note that $\Alt{4}$ and $\Sym{3}$ are not commutative, which implies that $S$ is also not commutative (see \cite[Proposition 3.1]{Semigroup_constructions_paper}). Furthermore, by observation of Figures~\ref{inv smg, Figure: commuting graph of S_3} and \ref{inv smg, Figure: commuting graph of A_4} it is easy to verify that $\cliquenumber{\commgraph{\Sym{3}}}=2$ and $\cliquenumber{\commgraph{\Alt{4}}}=3$, respectively. Therefore, by Corollary~\ref{0Union: clique number},
	\begin{displaymath}
		\cliquenumber{\commgraph{S}}=\cliquenumber{\commgraph{\Alt{4}}}+\parens{k-1}\cdot\cliquenumber{\commgraph{\Sym{3}}}=3 + \parens{k-1}\cdot 2=2k+1=n. \qedhere
	\end{displaymath}
\end{proof}

\begin{figure}[hbt]
	\begin{center}
		\begin{tikzpicture}
			
			\node[vertex] (2143) at (6.5,-1.5) {};
			\node[vertex] (3412) at (6.5,-3) {};
			\node[vertex] (4321) at (8,-2.25) {};
			\node[vertex] (3124) at (0,0) {};
			\node[vertex] (2314) at (1.5,0) {};
			\node[vertex] (3241) at (0,-1.5) {};
			\node[vertex] (4213) at (1.5,-1.5) {};
			\node[vertex] (1423) at (0,-3) {};
			\node[vertex] (1342) at (1.5,-3) {};
			\node[vertex] (2431) at (0,-4.5) {};
			\node[vertex] (4132) at (1.5,-4.5) {};

			\node[anchor=south] at (2143) {$\begin{pmatrix}1&2&3&4\\2&1&4&3\end{pmatrix}$};
			\node[anchor=north] at (3412) {$\begin{pmatrix}1&2&3&4\\3&4&1&2\end{pmatrix}$};
			\node[anchor=west] at (4321) {$\begin{pmatrix}1&2&3&4\\4&3&2&1\end{pmatrix}$};
			\node[anchor=east] at (3124) {$\begin{pmatrix}1&2&3&4\\3&1&2&4\end{pmatrix}$};
			\node[anchor=west] at (2314) {$\begin{pmatrix}1&2&3&4\\2&3&1&4\end{pmatrix}$};
			\node[anchor=east] at (3241) {$\begin{pmatrix}1&2&3&4\\3&2&4&1\end{pmatrix}$};
			\node[anchor=west] at (4213) {$\begin{pmatrix}1&2&3&4\\4&2&1&3\end{pmatrix}$};
			\node[anchor=east] at (1423) {$\begin{pmatrix}1&2&3&4\\1&4&2&3\end{pmatrix}$};
			\node[anchor=west] at (1342) {$\begin{pmatrix}1&2&3&4\\1&3&4&2\end{pmatrix}$};
			\node[anchor=east] at (2431) {$\begin{pmatrix}1&2&3&4\\2&4&3&1\end{pmatrix}$};
			\node[anchor=west] at (4132) {$\begin{pmatrix}1&2&3&4\\4&1&3&2\end{pmatrix}$};

			\draw[edge] (2143)--(3412);
			\draw[edge] (2143)--(4321);
			\draw[edge] (4321)--(3412);
			\draw[edge] (3124)--(2314);
			\draw[edge] (3241)--(4213);
			\draw[edge] (1423)--(1342);
			\draw[edge] (2431)--(4132);
			
		\end{tikzpicture}
	\end{center}
	\caption{Commuting graph of the alternating group $\Alt{4}$ on $\set{1,2,3,4}$.}
	\label{inv smg, Figure: commuting graph of A_4}
\end{figure}
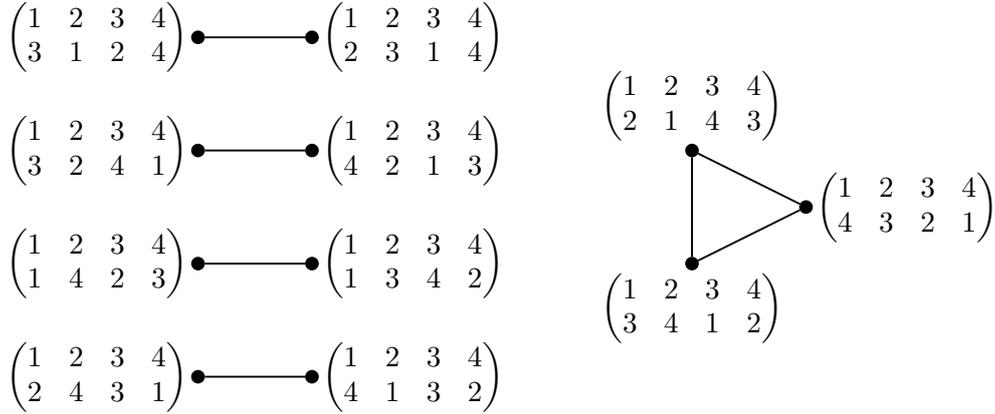

It follows from Theorem~\ref{inv smg: clique number Clifford} that all positive integers greater than $1$ arise as the clique number of some finite non-commutative inverse semigroup. That statement raises the question of whether $1$ is a possible value as well, that is, if it is possible for the commuting graph of an inverse semigroup to be a null graph --- Theorem~\ref{inv smg: clique number different from 1} shows us that it is not. We observe that the theorem also implies that $1$ is also not a possible value for the clique numbers of commuting graphs of groups and Clifford semigroups.

\begin{theorem}\label{inv smg: clique number different from 1}
	There is no finite non-commutative inverse semigroups whose commuting graph has clique number equal to $1$.
\end{theorem}

\begin{proof}
	Let $S$ be a finite non-commutative inverse semigroup. Our aim is to prove that $\cliquenumber{\commgraph{S}}\geqslant 2$.
	
	\smallskip
	
	\textit{Case 1:} Suppose that $S$ is a Clifford semigroup. Due to the fact that $S$ is not commutative, there exist $x,y\in S$ such that $xy\neq yx$.
	
	\smallskip
	
	\textsc{Sub-case 1:} Assume that $x\neq x^{-1}$ or $y\neq y^{-1}$. We have, by part 1 of Lemma~\ref{inv smg: Clifford smg, lemma girth}, that $x^{-1}y\neq yx^{-1}$ and $xy^{-1}\neq y^{-1}x$, which implies that $x^{-1}, y^{-1}\in S\setminus\centre{S}$ (that is, $x^{-1}$ and $y^{-1}$ are vertices of $\commgraph{S}$). Since $S$ is a Clifford semigroup, then we have that $xx^{-1}=x^{-1}x$ and $yy^{-1}=y^{-1}y$ (by \eqref{preli: Clifford smg, x commutes with inverse}). Since $x\neq x^{-1}$ or $y\neq y^{-1}$, then we can conclude that $x$ is adjacent to $x^{-1}$ or $y$ is adjacent to $y^{-1}$. Thus $\cliquenumber{\commgraph{S}}\geqslant 2$.
	
	\smallskip
	
	\textsc{Sub-case 2:} Assume that $x=x^{-1}$ and $y=y^{-1}$. It follows from \eqref{preli: inv smg, inverse product} that $xy\neq yx =y^{-1}x^{-1}=\parens{xy}^{-1}$, and it follows from \eqref{preli: Clifford smg, x commutes with inverse} that $\parens{xy}\parens{xy}^{-1}=\parens{xy}^{-1}\parens{xy}$. Additionally, part 2 of Lemma~\ref{inv smg: Clifford smg, lemma girth} allows us to conclude that $x\parens{xy}\neq \parens{xy}x$ and, consequently, part 1 of Lemma~\ref{inv smg: Clifford smg, lemma girth} ensures that we also have $x\parens{xy}^{-1}\neq \parens{xy}^{-1}x$. Hence $xy,\parens{xy}^{-1}\in S\setminus\centre{S}$ and $xy$ is adjacent to $\parens{xy}^{-1}$. Therefore $\cliquenumber{\commgraph{S}}\geqslant 2$.
	
	\smallskip
	
	\textit{Case 2:} Suppose that $S$ is not a Clifford semigroup. Since $S$ is an inverse semigroup, we can assume, by the Vagner--Preston Theorem~\ref{preli: vagner--preston theorem inv}, that $S$ is an inverse subsemigroup of $\psym{X}$, for some finite set $X$. Additionally, we know that $S$ contains a non-central idempotent (because $S$ is not a Clifford semigroup) and, if we use part 1 of Lemma~\ref{inv smg: lemma girth inv smg}, followed by part 2, we can conclude that there exists $\alpha\in S$ such that $\alpha\alpha^{-1}\neq\alpha^{-1}\alpha$ and $\alpha\alpha^{-1},\alpha^{-1}\alpha\in S\setminus\centre{S}$. Moreover, $\alpha\alpha^{-1}$ and $\alpha^{-1}\alpha$ commute since they are idempotents. Therefore, they are adjacent vertices of $\commgraph{S}$ and, consequently, $\cliquenumber{\commgraph{S}}\geqslant 2$.
\end{proof}

%Although there are no groups/Clifford semigroups/inverse semigroups whose commuting graph has clique number equal to $1$, the same does not apply to completely regular semigroups. In fact, it was proved in \cite[Corollary 4.1]{Completely_simple_semigroups_paper} that $1$ is a possible value for the clique number of the commuting graph of a completely simple semigroup (which is a completely regular semigroup).

\section{Possible values for the chromatic number of commuting graphs}\label{sec: chromatic number inv smg}

The aim of this section is to determine which positive integers arise as the chromatic number of the commuting graph of some group/Clifford semigroup/inverse semigroup/completely regular semigroup. In \cite[Corollary 4.4]{Completely_simple_semigroups_paper} the present author proved that the set of possible values for the chromatic number of the commuting graph of a completely simple semigroup is $\mathbb{N}$. Hence we can conclude that $\mathbb{N}$ is also the set of possible values for the chromatic number of the commuting graph of a completely regular semigroup. This means we only need to study this problem for groups, Clifford semigroups and inverse semigroups.

These results rely on two results that were obtained in \cite{Semigroup_constructions_paper} and which we present below.

\begin{corollary}\cite[Theorem 3.2 and Corollary 3.5]{Semigroup_constructions_paper}\label{0Union: chromatic number}
	Let $n\in\mathbb{N}\setminus\set{1}$ and let $S$ be the zero-union of the finite semigroups $S_1,\ldots,S_n$. Let $\NC=\{i\in\Xn: S_i \text{ is not}\allowbreak \text{commutative}\}$ and suppose that $\NC\neq\emptyset$. We have that $\chromaticnumber{\commgraph{S}}=\sum_{i\in\NC}\chromaticnumber{\commgraph{S_i}}$.
\end{corollary}

\begin{theorem}\label{direct prod: chromatic number}
	Let $n\in\mathbb{N}\setminus\set{1}$ and let $S$ be the direct product of the finite semigroups $S_1,\ldots,S_n$. Let $C=\gset{i\in\Xn}{S_i \text{ is commutative}}$ and let $\NC=\gset{i\in\Xn}{S_i \text{ is not commutative}}$ and suppose that $\NC\neq\emptyset$. We have that
	\begin{displaymath}
		\chromaticnumber{\commgraph{S}}\leqslant\parens[\bigg]{\,\prod_{i\in C}\abs{S_i}}\parens[\bigg]{\,\prod_{i\in \NC}{\parens[\big]{\abs{\centre{S_i}}+\chromaticnumber{\commgraph{S_i}}}}} -\prod_{i=1}^n{\abs{\centre{S_i}}}.
	\end{displaymath}
\end{theorem}

We start by proving that for each even number there is a group whose commuting graph has chromatic number equal to that even number.

\begin{theorem}\label{inv smg: chromatic number groups}
	For each $n\in\mathbb{N}$, there is a (finite non-abelian) group whose commuting graph has chromatic number equal to $2n$.
\end{theorem}

\begin{proof}
	Let $n\in\mathbb{N}$. Just as in the proof of Theorem~\ref{inv smg: clique number groups}, we consider the group $\Sym{3}\times C_n$. In Figure~\ref{inv smg, Figure: commuting graph of S_3} we can easily verify that $\chromaticnumber{\commgraph{\Sym{3}}}=2$. Moreover, $\abs{\centre{\Sym{3}}}=1$ and $\abs{\centre{C_n}}=\abs{C_n}=n$. Then it follows from Theorem~\ref{direct prod: chromatic number} that
	\begin{align*}
		\chromaticnumber{\commgraph{\Sym{3}\times C_n}}&\leqslant\parens[\big]{\abs{\centre{\Sym{3}}}+\chromaticnumber{\commgraph{\Sym{3}}}}\cdot\abs{C_n}-\parens{\abs{\centre{\Sym{3}}}\cdot\abs{\centre{C_n}}}\\
		&=3n-n\\
		&=2n.
	\end{align*}
	Furthermore, the clique number is a lower bound for the chromatic number. Thus, by the proof of Theorem~\ref{inv smg: clique number groups}, we also have $\chromaticnumber{\commgraph{\Sym{3}\times C_n}}\geqslant \cliquenumber{\commgraph{\Sym{3}\times C_n}}=2n$, which concludes the proof.
\end{proof}

In Theorem~\ref{inv smg: chromatic number Clifford} we will see that every positive integer greater than $1$ is the chromatic number of the commuting graph of some Clifford semigroup. Since Clifford semigroups are inverse and completely regular, then this result also holds for inverse and completely regular semigroups. Moreover, Corollary~\ref{inv smg: chromatic number different from 1} guarantees that $1$ is not a possible value for the chromatic number of the commuting graph of any inverse semigroup (and, consequently, Clifford semigroup and group).

\begin{theorem}\label{inv smg: chromatic number Clifford}
	For each $n\in\mathbb{N}$ such that $n\geqslant 2$, there is a (finite non-commutative) Clifford semigroup whose commuting graph has chromatic number equal to $n$.
\end{theorem}

\begin{proof}
	Assume that $n$ is even. It follows from Theorem~\ref{inv smg: chromatic number groups} that there exists a finite non-abelian group (and, consequently, a Clifford semigroup) $G$ such that $\chromaticnumber{\commgraph{G}}=n$.
	
	Assume that $n=3$. We have that $\Alt{4}$ is a Clifford semigroup (because it is a group) and it is straightforward from Figure~\ref{inv smg, Figure: commuting graph of A_4} that $\chromaticnumber{\commgraph{\Alt{4}}}=3=n$.
	
	Now assume that $n$ is odd and that $n\geqslant 5$. Let $k\in\mathbb{N}$ be such that $n=2k+1$. Then $k\geqslant 2$. Let $S$ be the zero-union of $\Alt{4}$ with $k-1$ copies of $\Sym{3}$. Due to the fact that $\Alt{4}$ and $\Sym{3}$ are Clifford semigroups (because they are groups), and by Proposition~\ref{inv smg: zero-union Clifford smg is Clifford smg}, we have that $S$ is a Clifford semigroup. In addition, $S$ is not commutative because $\Alt{4}$ and $\Sym{3}$ are not commutative (see \cite[Proposition 3.1]{Semigroup_constructions_paper}). Furthermore, we have $\chromaticnumber{\commgraph{\Sym{3}}}=2$ (which can be verified in Figure~\ref{inv smg, Figure: commuting graph of S_3}) and $\chromaticnumber{\commgraph{\Alt{4}}}=3$ (which can be verified in Figure~\ref{inv smg, Figure: commuting graph of A_4}). Thus Corollary~\ref{0Union: chromatic number} guarantees that \begin{displaymath}
		\chromaticnumber{\commgraph{S}}=\chromaticnumber{\commgraph{\Alt{4}}}+\parens{k-1}\cdot\chromaticnumber{\commgraph{\Sym{3}}}=3 + \parens{k-1}\cdot 2=2k+1=n. \qedhere
	\end{displaymath}
\end{proof}

\begin{corollary}\label{inv smg: chromatic number different from 1}
	There is no finite non-commutative inverse semigroups whose commuting graph has chromatic number equal to $1$.
\end{corollary}

\begin{proof}
	The result is an immediate consequence of Theorem~\ref{inv smg: clique number different from 1} and the fact that the clique number of a simple graph is a lower bound for its chromatic number.
\end{proof}

\section{Possible values for the knit degree}\label{sec: knit degree inv smg}

In \cite[Corollary 4.5]{Completely_simple_semigroups_paper} the present author showed that the commuting graphs of groups have no left paths. The first result we prove in this section states that the commuting graphs of Clifford semigroups also have no left paths. Furthermore, we will see that the commuting graphs of inverse semigroups can have left paths and that $1$ is a possible value for the knit degrees of inverse semigroups. Finally, we will see that, if the commuting graph of a (finite non-commutative) completely regular semigroup has at least one left path, then the knit degree of that semigroup must be at least $2$.

\begin{theorem}\label{inv smg: knit degree}
	Let $S$ be a finite non-commutative Clifford semigroup. Then $\commgraph{S}$ contains no left paths.
\end{theorem}

\begin{proof}
	Suppose, with the aim of obtaining a contradiction, that $\commgraph{S}$ contains left paths. Let $x_1-x_2-\cdots-x_n$ be a left path in $\commgraph{G}$. Then we have
	\begin{align*}
		x_1&=x_1x_1^{-1}x_1 &\bracks{\text{by \eqref{preli: Clifford smg, regular}}} \\
		&=\parens{x_1x_1^{-1}}x_1 &\bracks{\text{rearranging parentheses}} \\
		&=\parens{x_1^{-1}x_1}x_1 &\bracks{\text{by \eqref{preli: Clifford smg, x commutes with inverse}}} \\
		&=x_1^{-1}\parens{x_1x_1} &\bracks{\text{rearranging parentheses}} \\
		&=x_1^{-1}\parens{x_nx_1} &\bracks{\text{because } x_1x_1=x_nx_1} \\
		&=x_1^{-1}\parens{\parens{x_nx_n^{-1}x_n}x_1} &\bracks{\text{by \eqref{preli: Clifford smg, regular}}} \\
		&=x_1^{-1}\parens{x_nx_n^{-1}}x_nx_1 &\bracks{\text{rearranging parentheses}} \\
		&=x_1^{-1}x_nx_1\parens{x_nx_n^{-1}} &\bracks{\text{because } x_nx_n^{-1}\in\idemp{S}\cap\centre{S}}\\
		&=x_1^{-1}\parens{x_nx_1}x_nx_n^{-1} &\bracks{\text{rearranging parentheses}}\\
		&=x_1^{-1}\parens{x_1x_1}x_nx_n^{-1} &\bracks{\text{because } x_1x_1=x_nx_1} \\
		&=\parens{x_1^{-1}x_1}x_1x_nx_n^{-1} &\bracks{\text{rearranging parentheses}} \\
		&=\parens{x_1x_1^{-1}}x_1x_nx_n^{-1} &\bracks{\text{by \eqref{preli: Clifford smg, x commutes with inverse}}} \\
		&=\parens{x_1x_1^{-1}x_1}x_nx_n^{-1} &\bracks{\text{rearranging parentheses}} \\
		&=\parens{x_1x_n}x_n^{-1} &\bracks{\text{by \eqref{preli: Clifford smg, regular}}}\\
		&=\parens{x_nx_n}x_n^{-1} &\bracks{\text{because } x_1x_n=x_nx_n} \\
		&=x_n\parens{x_nx_n^{-1}} &\bracks{\text{rearranging parentheses}} \\
		&=x_n\parens{x_n^{-1}x_n} &\bracks{\text{by \eqref{preli: Clifford smg, x commutes with inverse}}} \\
		&=x_n, &\bracks{\text{by \eqref{preli: Clifford smg, regular}}}
	\end{align*}
	which is a contradiction. Thus $\commgraph{S}$ contains no left paths.
\end{proof}

Even though commuting graphs of Clifford semigroups contain no left paths (that is, there are no possible values for the knit degree of Clifford semigroups), there exist inverse semigroups whose commuting graphs have left paths. In the next  proposition we will see that, for a finite set $X$ such that $\abs{X}\geqslant 3$, the commuting graph of the symmetric inverse semigroup $\psym{X}$ on $X$ has left paths. We note that $\psym{X}$ is not commutative if and only if $\abs{X}\geqslant 2$ and, consequently, the commuting graph of $\psym{X}$ is only defined when $\abs{X}\geqslant 2$.

\begin{proposition}\label{P(X): knit degree P(X) I(X)}
	Suppose that $\abs{X}\geqslant 2$. We have that $\commgraph{\psym{X}}$ contains left paths if and only if $\abs{X}\geqslant 3$, in which case $\knitdegree{\psym{X}}=1$.
\end{proposition}

\begin{proof}
	We divide the proof into two parts. The first part is dedicated to verifying that, when $\abs{X}=2$, the graph $\commgraph{\psym{X}}$ does not contain left paths. The second part is used to establish that, when $\abs{X}\geqslant 3$, $\commgraph{\psym{X}}$ contains at least one left path of length $1$.
	
	\medskip
	
	\textbf{Part 1.} Suppose that $\abs{X}=2$ and assume that $X=\set{x_1,x_2}$. We know that left paths must have length at least $1$. By observation of Figure~\ref{P(X), Figure: commuting graph of P_2 and I_2} we immediately conclude that the unique non-trivial path in $\commgraph{\psym{X}}$ is
	\begin{displaymath}
		\begin{pmatrix}
			x_1\\x_1
		\end{pmatrix}-\begin{pmatrix}
			x_2\\x_2
		\end{pmatrix}.
	\end{displaymath}
	However
	\begin{displaymath}
		\begin{pmatrix}
			x_1\\x_1
		\end{pmatrix}\begin{pmatrix}
			x_1\\x_1
		\end{pmatrix}
		=\begin{pmatrix}
			x_1\\x_1
		\end{pmatrix}\neq\emptyset
		=\begin{pmatrix}
			x_2\\x_2
		\end{pmatrix}\begin{pmatrix}
			x_1\\x_1
		\end{pmatrix},
	\end{displaymath}
	which implies that the path in question is not a left path. Since $\commgraph{\psym{X}}$ contain no other non-trivial paths, we can conclude that $\commgraph{\psym{X}}$ contain no left paths.
	
	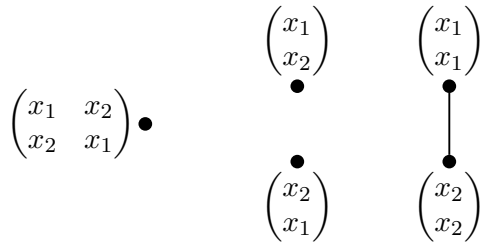
\begin{figure}[hbt]
		\begin{center}
			\begin{tikzpicture}
				\node[vertex] (21) at (2,-0.5) {};
				\node[vertex] (1) at (6,0) {};
				\node[vertex] (2) at (6,-1) {};
				\node[vertex] (a) at (4,0) {};
				\node[vertex] (b) at (4,-1) {};
				
				\node[anchor=east] at (21) {$\begin{pmatrix}x_1&x_2 \\ x_2&x_1\end{pmatrix}$};
				\node[anchor=south] at (1) {$\begin{pmatrix}x_1 \\ x_1\end{pmatrix}$};
				\node[anchor=north] at (2) {$\begin{pmatrix}x_2 \\ x_2\end{pmatrix}$};
				\node[anchor=south] at (a) {$\begin{pmatrix}x_1 \\ x_2\end{pmatrix}$};
				\node[anchor=north] at (b) {$\begin{pmatrix}x_2 \\ x_1\end{pmatrix}$};
				
				\draw[edge] (1) -- (2);
				
			\end{tikzpicture}
		\end{center}
		\caption{Commuting graph of the symmetric inverse semigroup $\psym{\set{x_1,x_2}}$ over $\set{x_1,x_2}$.}
		\label{P(X), Figure: commuting graph of P_2 and I_2}
	\end{figure}
	
	\medskip
	
	\textbf{Part 2.} Suppose that $\abs{X}\geqslant 3$. Then there exist pairwise distinct $x_1,x_2,x_3\in X$. Let
	\begin{displaymath}
		\alpha=\begin{pmatrix}
			x_1\\x_2
		\end{pmatrix} \quad \text{and} \quad
		\beta=\begin{pmatrix}
			x_1\\x_3
		\end{pmatrix}.
	\end{displaymath}
	We have that $\alpha\beta=\emptyset=\beta\alpha$, which implies that $\alpha-\beta$ is a path in $\commgraph{\psym{X}}$. Additionally, we have that $\alpha\alpha=\emptyset=\beta\alpha$ and $\alpha\beta=\emptyset=\beta\beta$. Therefore $\alpha-\beta$ is a left path (of length $1$) in $\commgraph{\psym{X}}$ and, thus, $\knitdegree{\psym{X}}=1$.
\end{proof}

\begin{theorem}\label{com reg: knit degree}
	There is no finite non-commutative completely regular semigroup whose knit degree is equal to $1$.
\end{theorem}

\begin{proof}
	Assume, with the aim of obtaining a contradiction, that there exists a finite non-commutative completely regular semigroup $S$ such that $\knitdegree{S}=1$. Then $\commgraph{S}$ contains at least one left path of length $1$. Let $x-y$ be one of those left paths. We have that $x\neq y$ and $xy=yx$ and $x^2=yx$ and $y^2=xy$.
	
	It follows from Theorem~\ref{preli: comp reg, semillatice comp simple} that there exist a semilattice $Y$ and completely simple subsemigroups $S_{\alpha}$ of $S$, where $\alpha\in Y$, such that $\set{S_{\alpha}}_{\alpha\in Y}$ is a partition of $S$ and $S_{\alpha}S_{\beta}\subseteq S_{\alpha\sqcap\beta}$ for all $\alpha,\beta\in S$.
	
	Let $\alpha,\beta\in Y$ be such that $x\in S_\alpha$ and $y\in S_\beta$. We aim to prove that $\alpha=\beta$. We have that $x^2\in S_\alpha S_\alpha\subseteq S_{\alpha\sqcap\alpha}=S_\alpha$ and $yx\in S_\beta S_\alpha \subseteq S_{\beta\sqcap\alpha}$ and $y^2\in S_\beta S_\beta\subseteq S_{\beta\sqcap\beta}=S_\beta$ and $xy\in S_\alpha S_\beta\subseteq S_{\alpha\sqcap\beta}=S_{\beta\sqcap\alpha}$. Then $x^2=yx\in S_\alpha\cap S_{\beta\sqcap\alpha}$ and $y^2=xy\in S_\beta\cap S_{\beta\sqcap\alpha}$. Consequently, we have that $S_\alpha=S_{\beta\sqcap\alpha}$ and $S_\beta=S_{\beta\sqcap\alpha}$, which implies that $\alpha=\beta\sqcap\alpha=\beta$. Hence $x,y\in S_\alpha$.
	
	We have that $S_\alpha$ is a finite completely simple semigroup. Hence there exist a (finite) group $G$, (finite) index sets $I$ and $\Lambda$ and a $\Lambda\times I$ matrix $P$ whose entries are elements of $G$ such that $S_\alpha\simeq\Rees{G}{I}{\Lambda}{P}$. We have two possibilities: either $I$ and $\Lambda$ are both singletons, or at least one of $I$ and $\Lambda$ is not a singleton. We will see that both cases lead to a contradiction.
	
	\smallskip
	
	\textit{Case 1:} Assume that $I$ and $\Lambda$ are singletons. Then $\Rees{G}{I}{\Lambda}{P}\simeq G$ (see, for example, \cite[Exercise 4.1]{Nine_chapters_Cain}) and, consequently, $S_\alpha$ is a group. Let $e\in S_\alpha$ be the identity of $S_\alpha$ and let $x',y'\in S_\alpha$ be the inverses of $x$ and $y$, respectively, in $S_\alpha$. Thus, since we also have $y^2=xy$, we obtain
	\begin{displaymath}
		x=xe=x\parens{yy'}=\parens{xy}y'=y^2y'=y\parens{yy'}=ye=y,
	\end{displaymath}
	which is a contradiction.
	
	\smallskip
	
	\textit{Case 2:} Assume that $I$ is not a singleton or $\Lambda$ is not a singleton. We have that $\centre{S_\alpha}=\centre{\Rees{G}{I}{\Lambda}{P}}=\emptyset$ (see \cite[Proposition 3.1]{Completely_simple_semigroups_paper}). Hence $S_\alpha$ is not commutative and so its commuting graph is defined. Moreover, we also have that $x,y\in S_\alpha\setminus\centre{S_\alpha}$, which implies that $x-y$ is a left path in $\commgraph{S_\alpha}$. However, $S_\alpha$ is a finite non-commutative completely simple semigroup and, by \cite[Corollary 4.5]{Completely_simple_semigroups_paper}, finite non-commutative completely simple semigroups have no left paths. This means we have reached a contradiction.
	
	\smallskip
	
	Since in both cases we obtained a contradiction, then we can conclude that there is no finite non-commutative completely regular semigroup whose knit degree is equal to $1$.
\end{proof}

\section{Open problems}

In this section we discuss some problems concerning commuting graphs of groups, inverse semigroups and completely regular semigroups that remain unsolved.

We saw in Theorems~\ref{inv smg: clique number groups} and \ref{inv smg: chromatic number groups} that each even positive integer arises as the the clique number and the chromatic number of the commuting graph of some group, and we saw in Theorems~\ref{inv smg: clique number different from 1} and \ref{inv smg: chromatic number different from 1} that $1$ is not a possible value for neither of the properties. We also noted that there exist groups such that the clique/chromatic numbers of their commuting graphs is an odd number. This bring us to the first open problem:

\begin{problem}
    For each positive integer $n$, is there a (finite non-abelian) group whose commuting graph has clique/chromatic number equal to $2n+1$?
\end{problem}
	
It follows from Proposition~\ref{P(X): knit degree P(X) I(X)} that $1$ is a possible value for the knit degree of an inverse semigroup.

\begin{problem}
    Is there a (finite non-commutative) inverse semigroup whose knit degree greater than $1$?
\end{problem}

In Theorems~\ref{com reg: girth} we saw that all even integers greater than $2$ are possible values for the girth of the commuting graph of a completely regular semigroup. This leads to the following question:

\begin{problem}
    For each $n\in\mathbb{N}$ such that $n\geqslant 3$ is odd, is there a (finite non-commutative) completely regular semigroup whose commuting graph has girth equal to $n$?
\end{problem}

In \cite[Theorems 3.10 and 3.14]{Commuting_graph_T_X} it was proved that for each $n\in\mathbb{N}\setminus\set{1,3}$ there is a completely regular semigroup whose knit degree is equal to $n$. Moreover, in Theorem~\ref{com reg: knit degree} we established that $1$ is not a possible value for the knit degree of a completely regular semigroups. This leads to the following question:

\begin{problem}
    Is there a (finite non-commutative) completely regular semigroup whose knit degree is equal to $3$?
\end{problem}

We can also study this type of problems for other classes of semigroups:

\begin{problem}
    Determine the possible values for the properties of the commuting graph of a regular semigroup.
\end{problem}

    \bibliography{Bibliography} %\jobname

\begin{thebibliography}{AKK11}

\bibitem[ABK15]{Commuting_graph_I_X}
Jo{\~a}o Ara{\'u}jo, Wolfram Bentz, and Janusz Konieczny.
\newblock The commuting graph of the symmetric inverse semigroup.
\newblock {\em Israel Journal of Mathematics}, 207:103--149, 2015.
\newblock \href {https://doi.org/10.1007/s11856-015-1173-9} {\path{doi:10.1007/s11856-015-1173-9}}.

\bibitem[AKK11]{Commuting_graph_T_X}
João Araújo, Michael Kinyon, and Janusz Konieczny.
\newblock Minimal paths in the commuting graphs of semigroups.
\newblock {\em European Journal of Combinatorics}, 32(2):178--197, 2011.
\newblock \href {https://doi.org/10.1016/j.ejc.2010.09.004} {\path{doi:10.1016/j.ejc.2010.09.004}}.

\bibitem[Ber83]{Importance_commuting_graphs_1}
Edward~A. Bertram.
\newblock Some applications of graph theory to finite groups.
\newblock {\em Discrete Mathematics}, 44(1):31--43, 1983.
\newblock \href {https://doi.org/10.1016/0012-365X(83)90004-3} {\path{doi:10.1016/0012-365X(83)90004-3}}.

\bibitem[BG89]{Symmetric_group}
J.~M. Burns and B.~Goldsmith.
\newblock Maximal order abelian subgroups of symmetric groups.
\newblock {\em Bull. Lond. Math. Soc.}, 21(1):70--72, 1989.
\newblock \href {https://doi.org/10.1112/blms/21.1.70} {\path{doi:10.1112/blms/21.1.70}}.

\bibitem[BG16]{Graphs_that_arise_as_commuting_graphs_of_semigroups}
Tomer Bauer and Be’eri Greenfeld.
\newblock Commuting graphs of boundedly generated semigroups.
\newblock {\em European Journal of Combinatorics}, 56:40--45, 2016.
\newblock \href {https://doi.org/10.1016/j.ejc.2016.02.009} {\path{doi:10.1016/j.ejc.2016.02.009}}.

\bibitem[Cai12]{Nine_chapters_Cain}
Alan~J. Cain.
\newblock {\em Nine Chapters on the Semigroup Art: Lecture notes for a tour through semigroups}.
\newblock Porto {\&} Lisbon, 2012.
\newblock URL: \url{https://archive.org/details/cain_semigroups_a4_screen}.

\bibitem[Cut22]{Group_whose_commuting_graph_has_diameter_n}
Giovanni Cutolo.
\newblock On a construction by giudici and parker on commuting graphs of groups.
\newblock {\em Journal of Combinatorial Theory, Series A}, 192, 2022.
\newblock Article no. 105666.
\newblock \href {https://doi.org/10.1016/j.jcta.2022.105666} {\path{doi:10.1016/j.jcta.2022.105666}}.

\bibitem[DO11]{Diameter_commuting_graph_symmetric_group}
David Dolžan and Polona Oblak.
\newblock Commuting graphs of matrices over semirings.
\newblock {\em Linear Algebra and its Applications}, 435(7):1657--1665, 2011.
\newblock Special Issue dedicated to 1st Montreal Workshop.
\newblock \href {https://doi.org/10.1016/j.laa.2010.04.014} {\path{doi:10.1016/j.laa.2010.04.014}}.

\bibitem[Fis71]{Sporadic_simple_groups}
Bernd Fischer.
\newblock Finite groups generated by 3-transpositions. {I}.
\newblock {\em Inventiones mathematicae}, 13(3):232--246, 1971.
\newblock \href {https://doi.org/10.1007/BF01404633} {\path{doi:10.1007/BF01404633}}.

\bibitem[GK16]{Graphs_that_arise_as_commuting_graphs_of_semigroups_2}
Michael Giudici and Bojan Kuzma.
\newblock Realizability problem for commuting graphs.
\newblock {\em Journal of the Australian Mathematical Society}, 101(3):335–355, 2016.
\newblock \href {https://doi.org/10.1017/S1446788716000148} {\path{doi:10.1017/S1446788716000148}}.

\bibitem[GP13]{Group_whose_commuting_graph_has_diameter_greater_than_n}
Michael Giudici and Chris Parker.
\newblock There is no upper bound for the diameter of the commuting graph of a finite group.
\newblock {\em Journal of Combinatorial Theory, Series A}, 120(7):1600--1603, 2013.
\newblock \href {https://doi.org/10.1016/j.jcta.2013.05.008} {\path{doi:10.1016/j.jcta.2013.05.008}}.

\bibitem[IJ08]{Commuting_graph_symmetric_alternating_groups}
A.~Iranmanesh and A.~Jafarzadeh.
\newblock On the commuting graph associated with the symmetric and alternating groups.
\newblock {\em Journal of Algebra and Its Applications}, 7(1):129--146, 2008.
\newblock \href {https://doi.org/10.1142/S0219498808002710} {\path{doi:10.1142/S0219498808002710}}.

\bibitem[MP13]{Commuting_graphs_group_trivial_upper_bound_diameter_10}
G.~L. Morgan and C.~W. Parker.
\newblock The diameter of the commuting graph of a finite group with trivial centre.
\newblock {\em Journal of Algebra}, 393:41--59, 2013.
\newblock \href {https://doi.org/10.1016/j.jalgebra.2013.06.031} {\path{doi:10.1016/j.jalgebra.2013.06.031}}.

\bibitem[Pau25a]{Largest_commutative_T_X_P_X}
Tânia Paulista.
\newblock Characterizing the largest commutative (full and partial) transformation semigroups of certain types.
\newblock Preprint, arXiv: 2511.09495, 2025.
\newblock URL: \url{https://arxiv.org/abs/2511.09495}.

\bibitem[Pau25b]{Semigroup_constructions_paper}
Tânia Paulista.
\newblock Commuting graphs and semigroup constructions.
\newblock Preprint, arXiv: 2511.03003, 2025.
\newblock URL: \url{https://arxiv.org/abs/2511.03003}.

\bibitem[Pau25c]{Completely_0-simple_paper}
Tânia Paulista.
\newblock Commuting graphs of completely 0-simple semigroups.
\newblock Preprint, arXiv: 2510.23871, 2025.
\newblock URL: \url{https://arxiv.org/abs/2510.23871}.

\bibitem[Pau25d]{Completely_simple_semigroups_paper}
Tânia Paulista.
\newblock Commuting graphs of completely simple semigroups.
\newblock {\em Communications in Algebra}, 53(10):4215--4226, 2025.
\newblock \href {https://doi.org/10.1080/00927872.2025.2481079} {\path{doi:10.1080/00927872.2025.2481079}}.

\bibitem[Pau25e]{My_thesis}
Tânia Paulista.
\newblock {\em Commuting graphs of semigroups}.
\newblock PhD thesis, NOVA School of Science \& Technology, 2025.

\bibitem[Pau25f]{Diameter_P_X}
Tânia Paulista.
\newblock Diameters of commuting graphs of partial transformation semigroups.
\newblock In preparation, 2025.

\bibitem[RS01]{Importance_commuting_graphs_2}
Andrei~S. Rapinchuk and Yoav Segev.
\newblock Valuation-like maps and the congruence subgroup property.
\newblock {\em Invent. Math.}, 144(3):571--607, 2001.
\newblock \href {https://doi.org/10.1007/s002220100136} {\path{doi:10.1007/s002220100136}}.

\bibitem[RSS02]{Importance_commuting_graphs_3}
Andrei~S. Rapinchuk, Yoav Segev, and Gary~M. Seitz.
\newblock Finite quotients of the multiplicative group of a finite dimensional division algebra are solvable.
\newblock {\em J. Am. Math. Soc.}, 15(4):929--978, 2002.
\newblock \href {https://doi.org/10.1090/S0894-0347-02-00393-4} {\path{doi:10.1090/S0894-0347-02-00393-4}}.

\bibitem[Seg99]{Importance_commuting_graphs_5}
Yoav Segev.
\newblock On finite homomorphic images of the multiplicative group of a division algebra.
\newblock {\em Annals of Mathematics. Second Series}, 149(1):219--251, 1999.
\newblock URL: \url{http://eudml.org/doc/120207}.

\bibitem[Seg01]{Importance_commuting_graphs_6}
Yoav Segev.
\newblock The commuting graph of minimal nonsolvable groups.
\newblock {\em Geometriae Dedicata}, 88(1):55--66, 2001.
\newblock \href {https://doi.org/10.1023/A:1013180005982} {\path{doi:10.1023/A:1013180005982}}.

\bibitem[SS02]{Importance_commuting_graphs_4}
Yoav Segev and Gary~M. Seitz.
\newblock Anisotropic groups of type $a_n$ and the commuting graph of finite simple groups.
\newblock {\em Pacific Journal of Mathematics}, 202(1):125--225, 2002.
\newblock \href {https://doi.org/10.2140/pjm.2002.202.125} {\path{doi:10.2140/pjm.2002.202.125}}.

\bibitem[Vdo99]{Alternating_group}
E.~P. Vdovin.
\newblock Maximal orders of abelian subgroups in finite simple groups.
\newblock {\em Algebra and Logic}, 38(2):67--83, 1999.
\newblock \href {https://doi.org/10.1007/BF02671721} {\path{doi:10.1007/BF02671721}}.

\end{thebibliography}
\bibliographystyle{alphaurl}

\end{document}